\documentclass[11pt]{article}

\usepackage{bbold}
\usepackage{bbm}
\usepackage[dvips]{graphics} 
\usepackage[usenames]{color}

\usepackage{fig4tex} 
\usepackage{theorem}
\usepackage{amssymb, amsmath, amsbsy, amsfonts}
\usepackage{stmaryrd}
\usepackage{times}

\usepackage{graphics,graphicx}

\newtheorem{theorem}{Theorem}[section]
\newtheorem{lemma}[theorem]{Lemma}
\newtheorem{proposition}[theorem]{Proposition}

\newtheorem{corollary}[theorem]{Corollary}
{\theorembodyfont{\rmfamily}
\newtheorem{example}[theorem]{Example}
\newtheorem{remark}[theorem]{Remark}
}

\newcommand{\Z}{\mathbb{Z}}
\newcommand{\R}{\mathbb{R}}
\newcommand{\T}{\mathbb{T}}

\newcommand{\F}{\mathcal{F}}
\newcommand{\D}{\mathcal{D}}
\newcommand{\Om}{\Omega}
\newcommand{\la}{\lambda}
\newcommand{\Supp}{{\rm Supp}}
\newcommand{\conv}{{\rm Conv}}

\newcommand{\footnoteremember}[2]{\footnote{#2}\newcounter{#1}\setcounter{#1}{\value{footnote}}}

\def\qed{\hfill $\square$ \goodbreak \bigskip}
\def\eps{\varepsilon}

\title{Regularity and Singularities of \\Optimal Convex Shapes in the Plane}

\author{Jimmy Lamboley\footnoteremember{1}{CEREMADE, Universit\'e Paris-Dauphine, Place du Mar\'echal de Lattre de Tassigny, 75775 Paris, France}, Arian Novruzi\footnote{University of Ottawa, Department of Mathematics and Statistics, 585 King Edward, Ottawa, ON, K1N 6N5, Canada}, Michel Pierre\footnote{ENS Cachan Bretagne, IRMAR, UEB, av Robert Schuman, 35170 Bruz, France}}
\date{\today}

\begin{document}
\maketitle
\begin{abstract}

We focus here on the analysis of the regularity or singularity of solutions $\Om_{0}$ to shape optimization problems among convex planar sets, namely:
\begin{equation*}
J(\Om_{0})=\min\{J(\Om),\ \Om\ \textrm{convex},\ \Omega\in\mathcal S_{ad}\},
\end{equation*}
where $\mathcal S_{ad}$ is a set of 2-dimensional admissible shapes and $J:\mathcal{S}_{ad}\rightarrow\R$ is a shape functional. 

Our main goal is to obtain qualitative properties of these optimal shapes by using first and second order optimality conditions, including the infinite dimensional Lagrange multiplier due to the convexity constraint. We prove two types of results:
\begin{itemize}
\item[i)]
under a suitable convexity property of the functional $J$, we prove that $\Omega_0$ is a $W^{2,p}$-set, $p\in[1,\infty]$.
This result applies, for instance, with $p=\infty$ when the shape functional can be written as 
$
J(\Omega)=R(\Omega)+P(\Omega),
$
where 
$R(\Om)=F(|\Om|,E_{f}(\Om),\la_{1}(\Om))$ involves the area $|\Omega|$, the Dirichlet energy $E_f(\Omega)$ or the first eigenvalue  of the Laplace-Dirichlet operator $\lambda_1(\Omega)$, and 
$P(\Omega)$ is the perimeter of $\Om$,
\item[ii)]
under a suitable concavity assumption on the functional $J$, we prove that $\Omega_{0}$ is a polygon.
This result applies, for instance, when the functional is now written as
$
J(\Om)=R(\Omega)-P(\Omega),
$ with the same notations as above. 

\end{itemize}

\noindent{\it Keywords:\,} Shape optimization, convexity constraint, optimality conditions, regularity of free boundary.
\smallskip

\end{abstract}

\section{Introduction}\label{sect:intro}

The goal of this paper is to develop general and systematic tools to prove the regularity or the singularity of optimal shapes in shape optimization problems {\em among convex planar sets}, namely problems like:
\begin{equation}\label{eq:pb}
\min\{J(\Om),\ \Om\ \textrm{convex},\ \Omega\in\mathcal S_{ad}\},
\end{equation}
where $\mathcal S_{ad}$ is a set of admissible shapes among subsets of $\R^2$ and $J:\mathcal{S}_{ad}\rightarrow\R$ is a shape functional. 
Our main objective is to obtain qualitative properties of optimal shapes by exploiting first and second order optimality conditions on \eqref{eq:pb} where {\em the convexity constraint} is included through appropriate infinite dimensional Lagrange multipliers.

{\em Our approach is analytic} in the sense that convex sets are represented through adequate parametrizations and we work with  the corresponding "shape functionals" defined on {\em spaces of functions}. In particular, we will use the classical polar coordinates representation of convex sets as follows:
\begin{equation}\label{eq:Omega_u}
\Om_u:=\left\{(r,\theta)\in [0,\infty)\times\R\;;\;r<\frac{1}{u(\theta)}\right\},
\end{equation}
where $u$ is a positive  and $2\pi$-periodic function, often called ``gauge function of $\Omega_{u}$''. 
It is well-known that
\begin{equation}\label{eq:conv}
\Om_u\textrm{ is convex }\Longleftrightarrow u''+u\geq 0.
\end{equation}
Thus, Problem \eqref{eq:pb} may be transformed into the following: 
\begin{equation}\label{eq:pbu}
\min\left\{j(u):=J(\Om_u),\; u''+u\geq0,\; u\in\mathcal F_{ad}\right\},
\end{equation}
where $\F_{ad}$ is a space of $2\pi$-periodic functions which will be chosen appropriately to represent $\mathcal S_{ad}$ in (\ref{eq:pb}).\\

We obtain two families of results depending on whether $j$ is "of convex type" or "of concave type". In the first case, we prove regularity of the optimal shapes. In the second case, we prove that optimal shapes are polygons.
\begin{itemize}
\item[i)] {\em "Optimal shapes are regular":} under a suitable convexity property on the "main part" of the functional $j$, we prove that any solution $u_0$ of \eqref{eq:pbu} is $W^{2,p}$, which means that the curvature of $\partial\Omega_0=\partial\Omega_{u_0}$ is an $L^p$ function whereas it is a priori only a measure: see Theorems \ref{th:reg}, \ref{th:reg+constraint} and Corollary \ref{cor:reg}.
To that end, {\em we simply use the first optimality condition} for the problem \eqref{eq:pb}.

The functionals under consideration here are of the form $ J(\Omega)=R(\Omega)+C(\Omega)$, where $r(u):=R(\Omega_u)$ has an $L^p$-derivative and $C$ is like (\ref{eq:geometric}) below and satisfies a convexity condition. As a main example, we consider
$R(\Om)=F(|\Om|,E_{f}(\Om),\la_{1}(\Om))$ which depends on the area $|\Omega|$, on the Dirichlet energy $E_f(\Omega)$ and/or on the first eigenvalue $\lambda_1(\Omega)$ of the Laplace operator on $\Omega$ (with Dirichlet boundary conditions), and 
$C(\Omega)=P(\Omega)$ is its perimeter: see Section \ref{sect:examples-smooth}. In this case, we actually prove that the optimal shape is $W^{2,\infty}$ which means that the curvature is bounded.

\item[ii)] {\em "Optimal shapes are polygons":}
next, we prove that, under a suitable concavity assumption on the functional $j$, for any solution $u_0$ of 
\eqref{eq:pbu}, $u_0+u_0''$ is (locally) a finite sum of Dirac masses, so that $\Omega_{u_0}$ is (locally) a polygon: see Theorems \ref{th:conc}, \ref{th:conc+constraint} and Corollary \ref{cor:conc}. The proof of this result is {\em based on the second order optimality condition} for the problem \eqref{eq:pb}.
We apply this result to shape optimization problems where $
J(\Om)=R(\Omega)-P(\Omega)$ where $R(\Omega)=F(|\Om|,E_{f}(\Om),\la_{1}(\Om))$ with the same notations as above, see Section \ref{sect:examples-polygon}. This application involves some sharp  estimates on the second shape derivative of the energy which are interesting for themselves: see Section \ref{ssect:dir2}.
\end{itemize}

Our examples enlighten and exploit the fact that, in the context of shape optimization under convexity constraint, the perimeter is ``stronger'' than usual energies involving PDE, in terms of the influence on the qualitative properties of optimal shapes: if it appears in the energy as a positive term, it has a smoothing effect on optimal shapes, and on the opposite as a negative term, it leads to polygonal optimal shapes.\\

\noindent{\em Dual parametrization:} Since our results are stated for the analytic functionals \eqref{eq:pbu}, we may apply them to the dual parametrization of convex sets instead of the parametrization with the gauge function: each convex shape can also be associated to its support function $h_{\Om}(\theta)=\max\{x\cdot e^{i\theta},\; x\in\Omega\}, \theta\in\T$ and \eqref{eq:pb} again leads to the problem:
\begin{equation}\label{eq:pbh}
\min\left\{\widetilde{j}(h),\; h''+h\geq0,\; h\in\widetilde{\mathcal F_{ad}}\right\},
\end{equation}
where $\widetilde{j}(h):=J(\Om^h)$, $\Om^h$ being now the set whose support function is $h$, and $\widetilde{\F_{ad}}$ are all support functions of admissible shapes $\Om\in\mathcal S_{ad}$. 
In this framework, if $\widetilde{j}$ satisfies the suitable convexity property, the regularity result (i) above holds for $h_{0}$ minimizer of \eqref{eq:pbh}. However, this regularity does not imply  that the corresponding optimal shape $\Om_{0}:=\Om^{h_{0}}$ is regular, but it exactly means that $\Om_{0}$ is strictly convex: see Section \ref{sect:support}.

The situation is more similar to the gauge representation when exploiting the results (ii). Indeed, when they apply, they imply that the optimal shape is polygonal as well: see Remark \ref{rem:polyh}.\\

\noindent{\em Situation with respect to previous results:} The second family of results (ii) is an extension of previous results obtained in \cite{LN} by the two first authors for the specific following functionals of ``local type":
\begin{equation}\label{eq:geometric}
J(\Om_u)=\int_0^{2\pi} G\left(\theta,u(\theta),u'(\theta)\right) d\theta,
\end{equation}
where  $G=G(\theta,u,q): \T\times[0,\infty)\times \R\to \R$ is strictly concave in $q$.
Among these functionals, we find for instance the area $|\Omega|$, the perimeter $P(\Omega)$ or also the famous Newton's problem of the body of minimal resistance as studied by  T. Lachand-Robert and coauthors: see for example \cite{CLR,LRP} and see also \cite{LN,MC} for more examples arising in the operator theory. Actually, the techniques employed in \cite{LN}, and here as well for (ii), are inspired from those introduced in \cite{LRP}. The main novelty here in the results (ii) is that the functionals are not necessarily of the local form (\ref{eq:geometric}) and may include shape functionals defined through state functions which are solutions of partial differential equations (PDE). The ``concavity condition" is then expressed in a functional way through the coercivity of the second derivative in an adequate functional space : see Theorem \ref{th:conc}. In \cite{BFL}, a similar concavity phenomenon is used to get qualitative properties of minimizers in higher dimension, under assumptions about their regularity and convexity. We avoid here any assumption of this kind for the planar case.
 
 The general optimality conditions including the infinite dimensional Lagrange multipliers were also provided (and exploited) in the same paper \cite{LN}. They are revisited here in an $W^{1,\infty}$-context which is better adapted to our more general functionals (see e.g. Proposition \ref{prop:order1}).
 
 Similar arguments to those used here to obtain the first family of results (i) may also be found in \cite{C} where optimality conditions with convexity constraints are developed in an $N$-dimensional setting. They are exploited for several examples in dimension 1 (or in radial situations) to obtain $C^1$-regularity of the optimal shapes. With our approach here, we are able to reach $W^{2,\infty}$-regularity and this is valid for a rather general family of functionals.\\
 
\noindent{\em About a localization of the approach:} Let us mention that our two families of results may be mixed in the same functional: indeed, as often the case, it may be that the required convexity property for (i) is valid on some part of the boundary of the optimal shape, while the concavity property for (ii) is valid on the other part. Then, the techniques developed here may be locally applied to each part and we can obtain at the same time smooth and polygonal pieces in the boundary. However, as one expects, it remains difficult to understand the portion of the boundary which remains at the intersection of these two parts. We refer to Section \ref{sect:u'3} for more details.\\

To end this introduction, let us say that many questions are of interest in shape optimization among convex sets. Here, we try to exploit as much as possible analytical tools to obtain precise qualitative results for optimal shapes among convex planar sets. But many questions are left open in higher dimensions. Among them, and besides the Newton's problem already mentioned, we can quote the famous Mahler conjecture about the minimization of the so-called Mahler-product $|K||K^\circ|$ among symmetric convex bodies in $\R^d$ (see \cite{T}), which is of great interest in convex geometry and functional analysis, and the P\'olya-Szeg\"o conjecture about the minimization of the Newtonian capacity among convex bodies of $\R^3$ whose surface area is given (see for example \cite{CFG} and reference therein). 

This paper is structured as follows. In the following section we state our main results. In 
Section \ref{sect:W2p} we focus on the regularity result (i) and we apply it to some various examples. In Section \ref{sect:polygon}, we deal with problems leading to polygonal solutions (result (ii)), and we again consider in detail some classical examples. We conclude with some remarks and perspectives.

\section{Main results}\label{sect:results}

\subsection{Notations and problems}\label{sect:not}

We set $\T:=[0,2\pi)$. Throughout the paper, any function defined on $\T$ is considered as the restriction to $\T$ of a $2\pi$-periodic function on $\R$. We define $W^{1,\infty}(\T):=\{u\in W^{1,\infty}_{loc}(\R), u\textrm{ is }2\pi\textrm{-periodic}\}$, and similarly for any functional space. If $u\in W^{1,\infty}(\T)$, we say that $u''+u\geq0$ if
\begin{equation*}
\forall\; v\in W^{1,\infty}(\T)\textrm{ with }v\geq0,\;\;\int_{\T} \left(uv - u'v'\right)d\theta\geq0.
\end{equation*}
In this case, $u''+u$ is a nonnegative $2\pi$-periodic measure on $\R$ and finite on $[0,2\pi]$.

We denote by $\mathcal S_{ad}$ a class of open bounded sets in $\R^2$ (including constraints besides convexity). We will focus on two problems:
\begin{equation}\label{eq:minshape1}
\min\{J(\Om),\ \Omega\in\mathcal S_{ad},\ \Om\ \textrm{convex}\}, 
\end{equation}
\begin{equation}\label{eq:minshape2}
\min\{J(\Om),\ \Omega\in\mathcal S_{ad},\ \Om\ \textrm{convex}, M(\Om)=M_{0}\},
\end{equation}
where $J:\mathcal S_{ad}\to\R$ is referred as the energy and $M:\mathcal S_{ad}\to\R^d$ is an extra constraint ($M_{0}$ given in $\in\R^d$).

In order to analyze the regularity of an optimal shape, we transform these problems into minimization problems in a functional analytic setting as follows: choosing an origin $O$ and using parameterization \eqref{eq:Omega_u}, we define 
\begin{equation}\label{eq:Fad}
\mathcal F_{ad}:=\{u\in W^{1,\infty}(\T),\;\; \Omega_u\in \mathcal S_{ad}\},
\end{equation}
 the set of admissible gauge functions,
endowed with the $\|\cdot\|_{W^{1,\infty}(\T)}$-norm, and we assume that this set can be written
\begin{equation}\label{eq:Fad2}
\mathcal F_{ad}=\{u\in W^{1,\infty}(\T)\;/\; k_{1}\leq u\leq k_{2}\textrm{ and }u>0\},
\end{equation}
for some functions $k_{1},k_{2}:\T\to\overline{\R_{+}}$ respectively upper- and lower-semicontinuous (see Remark \ref{k1k2} below for this assumption).

 A simple calculus of the curvature shows that $\Om_u$ is convex if and only if $u''+u\geq0$. Moreover, the support of the measure $u''+u$ gives a parametrization of the ``strictly convex part'' of the boundary, and  a Dirac mass in this measure correspond to a corner of the associated shape; we have for instance that $\Om_{u}$ is a convex polygon if and only if $u''+u$ is a finite sum of positive Dirac masses.

If $\Om_{0}$ is a solution of problem \eqref{eq:minshape1} (resp. \eqref{eq:minshape2}), then its gauge function $u_{0}$ is respectively solution of: 
\begin{equation}\label{eq:min}
j(u_0)=\min\big\{j(u),\;\;\ u''+u\geq 0,\; u\in\mathcal{F}_{ad}\,\big\},
\end{equation}
\begin{equation}\label{eq:ming}
{\rm resp.}\;\;j(u_0)=\min\big\{j(u)\;/\; u\in\mathcal{F}_{ad}, u''+u\geq 0\textrm{ and }m(u)=M_{0}\big\}, \end{equation}
where $j:\mathcal F_{ad}\mapsto\mathbb R$,  $j(u):=J(\Omega_u)$, and $m:\F_{ad}\to\R$, $m(u)=M(\Omega_u)$.\\

Our main goal in this paper is the analysis of the convexity constraint. Thus, given an optimal shape $\Om_{0}$, we focus on the part of $\partial\Omega_0$ which does not saturate the other constraints defined by $\mathcal S_{ad}$. We therefore define, for $u_{0}\in \mathcal{F}_{ad}$ and $\Om_{0}=\Om_{u_{0}}$,
\begin{eqnarray}
\T_{in}
&:=&
\T_{in}(\mathcal{F}_{ad},u_{0})
=\{\theta\in\T\; /\; k_{1}(\theta)<u_{0}(\theta)<k_{2}(\theta)\},  \label{eq:inside}
\\
(\partial\Omega_0)_{in}
&:=&\left\{x\in\partial\Omega_0 \;/\; \exists \theta\in \T_{in}, x=\frac{1}{u_{0}(\theta)}(\cos\theta,\sin\theta)\right\}.
\label{eq:inside2}
\end{eqnarray}
See Example \ref{ex:inclusion} and Figure 1 for examples.

\begin{remark}\label{k1k2}
If $k_{1}$ or $k_{2}$ happened not to be semicontinuous, we could replace them by
$$\overline{k_{1}}=\inf\{k:\T\to\R\textrm{ continuous }, k\geq k_{1}\}, \quad \underline{k_{2}}=\sup\{k:\T\to\R\textrm{ continuous }, k\leq k_{2}\}$$
and we have
$$\{u\in W^{1,\infty}(\T)\;/\; {k_{1}}\leq u\leq {k_{2}}\}=\{u\in W^{1,\infty}(\T)\;/\; \overline{k_{1}}\leq u\leq \underline{k_{2}}\}.$$
Therefore, the assumptions on $k_{1}$ and $k_{2}$ are not restrictive.
Note that, thanks to the regularity of $u_{0},k_{1},k_{2}$, the set $\T_{in}$ is open.
\end{remark}

\begin{example}\label{ex:inclusion}
%
\begin{figure}[htbp]
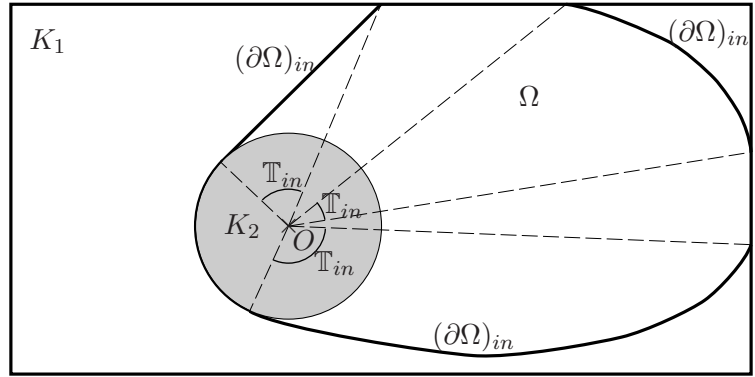

\figinit{0.7pt}
\figpt 0:(0,0)
\figpt 1:(200,100)
\figpt 2:(-200,100)
\figpt 3:(-200,-100)
\figpt 4:(200,-100)

\figpt 5:(-50,-20)

\figpt 6:(-86.5,14.5)
\figpt 7:(0,100)
\figpt 8:(100,100)

\figpt 9:(155,80)
\figpt 10:(180,60)
\figpt 11:(194,40)
\figpt 13:(200,20)

\figpt 14:(200,-30)
\figpt 15:(175,-60)
\figpt 16:(130,-80)
\figpt 17:(50,-90)
\figpt 18:(-60,-80)
\figpt 19:(-71,-66)
\figpt 20:(-100,0)


\figpt 21:(177,84)
\figpt 22:(-53,8)
\figpt 23:(-75,-20)
\figpt 24:(-180,80)
\figpt 25:(80,50)
\figpt 26:(50,-80)
\figpt 27:(-57,70)
\figpt 28:(-25,-40)
\figpt 29:(-21,-9)

\psbeginfig{}

\psset(fillmode=no,color=0)
\psset(dash=1)
\psset(width=1.2)
\psline[1,2,3,4,1]
\psline[6,7]
\psset curve(roundness=0.11)
\pscurve[7,8,9,10,11,13,14]
\pscurve[13,14,15,16,17,19,20]
\psset(width=0.8)
\pscirc 5(50)
\psset(fillmode=yes,color=0.8)
\pscirc 5(50)
\psset(fillmode=no,color=0)
\psset(width=0.8)
\psarccirc 5 ; 50.5(135,246)
\psline[7,8]
\psline[13,14]

\psset(width=0.5)
\psarccirc 5 ; 20(70,135)
\psarccirc 5 ; 20(10,40)
\psarccirc 5 ; 20(-115,-4)
\psset(width=0.3)
\psset(dash=2)
\psline[5,19]
\psline[5,8]
\psline[5,7]
\psline[5,6]
\psline[5,13]
\psline[5,14]

\psendfig
\figvisu{\figBoxA}{}{
\figwritec [5]{$\times$}
\figwritese 5:{$O$}(3)
\figwritec [21]{$(\partial\Omega)_{in}$}
\figwritec [22]{$\mathbb{T}_{in}$}
\figwritec [23]{$K_{2}$}
\figwritec [24]{$K_{1}$}
\figwritec [25]{$\Omega$}
\figwritec [26]{$(\partial\Omega)_{in}$}
\figwritec [27]{$(\partial\Omega)_{in}$}
\figwritec [28]{$\mathbb{T}_{in}$}
\figwritec [29]{$\mathbb{T}_{in}$}
}
\centerline{\box\figBoxA}
\caption{
Inclusion constraints}
\end{figure}
A frequent example for admissible shapes $\mathcal S_{ad}$ is:
$$\mathcal S_{ad}:=\left\{\Om\textrm{ bounded open set of }\R^2\;/\;K_{2}\subset\Om\subset K_{1}\right\},$$
where $K_{2}$ and $K_{1}$ are two given bounded open sets.
If for example $K_{1}$ and $K_{2}$ are starshaped with respect to a common point $O$, chosen as the origin, then
$$\mathcal{F}_{ad}=\{u\in W^{1,\infty}(\T)\; /\; k_{1}\leq u\leq k_{2}\},$$
where $k_{1}$, $k_{2}$ are the gauge functions of $K_{2}$ and $K_{1}$ respectively. In that case, given a set $\Om\in \mathcal{S}_{ad}$,
$$(\partial\Om)_{in}=\partial\Om\setminus{(\partial K_{1}\cup\partial{K_{2}})},$$
see Figure 1.

The analysis of the optimal shape around the set $\{\theta\;/\;u_0(\theta)\in\{k_{1}(\theta),k_{2}(\theta)\}\}=\T\setminus\T_{in}$, where the inclusion constraint is saturated, may require more efforts, see \cite{LN} for example. In this paper, we will not discuss this question.

Note that we can also consider the case $K_{2}=\emptyset$ and/or $K_{1}=\R^2$ with $k_{1}=0$ and $k_{2}=+\infty$.
\end{example}

\begin{example} With respect to the constraints $m,M$ in (\ref{eq:minshape2}), (\ref{eq:ming}), a classical example is the area constraint: 
$$
m(u):=
|\Om_{u}|=A_{0}\;\Longleftrightarrow\; \int_{\T}\frac{1}{2u^2}d\theta=A_{0},
$$
where $|\Omega|$ denotes the area of $\Omega$.
\end{example}

\subsection{The main results}
As explained in the introduction, Section \ref{sect:intro}, we will prove two types of results: they are described in the two following subsections.

\subsubsection{"Optimal shapes are smooth"}\label{sect:reg}

First we consider the problem \eqref{eq:minshape1} and its associated analytical version \eqref{eq:min}. We assume that $J(\Omega)=R(\Omega)+C(\Omega)$, $R$ satisfying some ``regularity'' assumption, and $C$ being written like in \eqref{eq:geometric}, and satisfying a convexity like property. More precisely:

\begin{theorem}\label{th:reg} 
Let $u_{0}>0$ be an optimal solution of \eqref{eq:min} with $\F_{ad}$ of the form \eqref{eq:Fad2} and
\begin{equation}\label{eq:j}
j(u):=r(u)+\displaystyle{\int_{\mathbb T} G\left(\theta,u(\theta),u'(\theta)\right) d\theta},
\quad
\forall u\in W^{1,\infty}(\Om)\cap\{u>0\},
\end{equation}
where $r$ and $G$ satisfy:
\begin{itemize}
\item[i)]
$r:W^{1,\infty}(\T)\to\R$ is $ C^1$ around $u_{0}$ and $G:(\theta,u,q)\in\T\times(0,\infty)\times\R\to\R$ is $ C^2$ around $\T\times u_{0}(\T)\times \conv(u_{0}'(\T))$, where $\conv(u_0'(\T))$ is the smallest (bounded) closed interval containing the values of the right- and left-derivatives $u_0'(\theta^+), u_0'(\theta^-), \theta\in\T$, 
\item[ii)]
$r'(u_{0})\in L^p(\T)$ for some $p\in[1,\infty]$,
\item[iii)]
$G_{qq}>0$ in $\T\times u_{0}(\T)\times \conv(u_{0}'(\T))$.
\end{itemize}
Then
$$u_{0}\in W^{2,p}(\T_{in}), \textrm{ where }\T_{in}\textrm{ is defined in }\eqref{eq:inside}.$$
\end{theorem}
See Section \ref{sect:proof-reg} for the proof, and Section \ref{sect:examples-smooth} for explicit examples.
\begin{remark}\label{rk:measure}
A $C^1$-regularity result has been proved for a similar problem with $r=0$ in \cite{C} with different boundary conditions, with a proof which is also based on first order optimality conditions. Here, for periodic boundary conditions (but this is not essential), we improve this result to the $ C^{1,1}$-regularity, and generalize it to the case of non-trivial $r$, which is of great interest for our applications. Let us also refer to \cite{CLR2} for a higher dimensional result.

Let us remark that the same result is valid, with the same proof, if we only assume that
$r'(u_{0})$ is the sum of a function in $L^p(\T)$ and of a nonpositive measure on $\T$.
\hfill$\Box$
\end{remark}

\noindent
We can also get a similar result for the equality constrained problem \eqref{eq:minshape2} and the associated problem
\eqref{eq:ming} as follows.
\begin{theorem}\label{th:reg+constraint} 
Let $u_{0}>0$ be an optimal solution of \eqref{eq:ming} with $j, \F_{ad}$ as in Theorem \ref{th:reg}, and $m:W^{1,\infty}\to\R^d$ a $ C^1$ function 
around $u_{0}$ with $m'(u_{0})\in (L^p(\T))^d$ onto.
Then
$$u_{0}\in W^{2,p}(\T_{in}).$$
\end{theorem}
See Section \ref{sect:proof-reg} for the proof, and Section \ref{sect:examples-smooth} for explicit examples.\\

\noindent
For a shape functional, using parametrization \eqref{eq:Omega_u},
Theorems \ref{th:reg} and \ref{th:reg+constraint} lead to the following.
\begin{corollary}\label{cor:reg} Let $\mathcal S_{ad}$ be a class of open sets in $\R^2$ such that $\F_{ad}:=\{u\;/\;\Om_{u}\in\mathcal S_{ad}\}$ is of the form \eqref{eq:Fad2} ($\Om_{u}$ is defined in \eqref{eq:Omega_u}), and let $J:\mathcal S_{ad}\to\R$ be a shape functional:\\
i) 
Let $\Om_0$ be an optimal shape for
problem \eqref{eq:minshape1}, and assume that $J=R+C$ with:
$$\forall u\in\mathcal{F}_{ad},\quad R(\Om_{u})=r(u)\textrm{ and }C(\Om_{u})=\int_{\T}G(\theta,u(\theta),u'(\theta))d\theta,$$
where $r$ and $G$ satisfy assumptions of Theorem \ref{th:reg} for some $p\in[1,\infty]$.
Then $(\partial\Omega_0)_{in}$, as defined in \eqref{eq:inside2}, is $\mathcal{C}^1$ and its curvature is in $L^p((\partial\Om_{0})_{in})$.\\
ii)
A similar results holds for the problem \eqref{eq:minshape2}, if $m(u)=M(\Om_{u})$ satisfies the hypotheses in Theorem \ref{th:reg+constraint}.
\end{corollary}

\begin{remark}
The results of this section are in an abstract analytical context, and do not depend on the characterization of the domain.
Therefore, one could consider the classical characterization of a convex body with its support function instead of the gauge function. In Section \ref{sect:support}, we give a geometrical interpretation of similar results associated to this parametrization.
\end{remark}

\subsubsection{"Optimal shapes are polygons"}\label{sssect:poly}

Our second result is a generalization of Theorem 2.1 from \cite{LN}. We give a sufficient condition on the shape functional $J$ so that any solution of \eqref{eq:pb} be a polygon. In \cite{LN}, the first two authors only consider shape functionals of local type like \eqref{eq:geometric}. 
The following results deal with non-local functionals, which allow a much larger class of applications, including shape functionals depending on a PDE.

\begin{theorem}\label{th:conc}
Let $u_{0}>0$ be a solution for \eqref{eq:min} with $\F_{ad}$ of the form \eqref{eq:Fad2}, and assume that $j:W^{1,\infty}(\T)\to\R$ is $\mathcal{C}^2$ around $u_{0}$ and satisfies (see Section \ref{sect:proof-conc} for definitions of $H^s$-(semi-)norms):
\begin{eqnarray}
&&
\exists s\in[0,1), \alpha>0, \beta,\gamma\in [0,\infty),\; \mbox{\it such that}\;\nonumber\\
&&
\forall v\in W^{1,\infty}(\T),\quad j''(u_0)(v,v)\leq -\alpha|v|^2_{H^1(\mathbb T)} + \gamma |v|_{H^1(\mathbb T)}\|v\|_{H^{s}(\mathbb T)}+\beta\|v\|^2_{H^{s}(\mathbb T)}.
\label{eq:coercivite}
\end{eqnarray}
If $I$ is a connected component of $\T_{in}$ (defined in \eqref{eq:inside}), then
$$
u_{0}''+u_{0}\textrm{ is a finite sum of Dirac masses in }I.
$$
\end{theorem}
\noindent
See Section \ref{sect:proof-conc} for a proof and Section \ref{sect:examples-polygon} for explicit examples.
\begin{remark}
We can even get an estimate of the number of Dirac masses in terms of $\alpha,\beta,\gamma$, see Remark \ref{rk:corner}.
\end{remark}
\begin{remark}
Theorem \ref{th:conc} remains true if \eqref{eq:coercivite} holds only for any 
$v$ such that (denoting $\mu=u_0''+u_0$):
$$\exists \,\varphi\in L^\infty(\T,\mu) \textrm{ with } v''+v=\varphi \,\mu.$$
Indeed, the proof of Theorem \ref{th:conc} uses only this kind of perturbations $v$ which preserve the convexity of the shape.
\hfill$\Box$
\end{remark}

As in Section \ref{sect:reg}, we can also handle the problem with an equality constraint as follows.
\begin{theorem}\label{th:conc+constraint} 
Let $u_{0}>0$ be any optimal solution of \eqref{eq:ming} with $j, \F_{ad}$ as in Theorem \ref{th:conc}, and the new assumptions:
\begin{eqnarray*}
j'(u_{0})\in \left( C^0(\T)\right)',\; 
&\mbox{\it and}&
m:W^{1,\infty}\to\R^d\; \textrm{is}\;  C^2\; \textrm{around}\; u_{0},\\\
&&
m'(u_{0})\in\left( C^0(\T)'\right)^d\;\ \textrm{ is\;onto,}\;\;
\|m''(u_{0})(v,v)\|\leq \beta'\|v\|^2_{H^{s}(\T)},\;\ \textrm{for some}\; \beta'\in \R.
\end{eqnarray*}
Then, if $I$ is a connected component of $\T_{in}$ (defined in \eqref{eq:inside}),
$$
u_{0}''+u_{0}\textrm{ is a finite sum of Dirac masses in }I.
$$
\end{theorem}
See Section \ref{sect:proof-conc} for the proof.\\

Again, using the parametrization \eqref{eq:Omega_u}, we get the following result.
\begin{corollary}\label{cor:conc}
Let $\mathcal S_{ad}$ be a class of open sets in $\R^2$ such that $\F_{ad}:=\{u\;/\;\Om_{u}\in\mathcal S_{ad}\}$ is of the form \eqref{eq:Fad2}, 
$\Om_0$ be an optimal shape for the problem \eqref{eq:minshape1} (or \eqref{eq:minshape2} for the constrained problem), and assume that $j: u\in\F_{ad}\mapsto J(\Om_{u})$ satisfies assumptions of Theorem \ref{th:conc} (and  
$m: u\in\F_{ad}\mapsto M(\Om_{u})$ satisfies assumption in Theorem \ref{th:conc+constraint} in the case of the constrained problem).
Then:
$$
\textrm{each connected component of }(\partial\Omega_0)_{in}\textrm{ is polygonal}.
$$
\end{corollary}

\begin{remark}\label{rem:polyh}
When one uses the parametrization of convex sets by the gauge function $u$, $\Om_{u}$ is a polygon if and only if $u''+u$ is a sum of Dirac masses. When parametrizing $\Omega$ with the support function as in Section \ref{sect:support}, one has the same characterization. Therefore, the results of this section hold if we work with the optimization problems as in Section \ref{sect:support}. 

\end{remark}

\section{Shape functionals containing a local-convex term}\label{sect:W2p}

In this section, we give the proof of the results in Section \ref{sect:reg}, that is to say regularity results for solutions of \eqref{eq:min} or \eqref{eq:ming}. Using the parametrization \eqref{eq:Omega_u}, since the regularity of a shape and of its gauge functions are the same, we consider several applications of regularity for optimal shapes to classical examples of energies. We conclude with a few remarks about the application of our results when we use another parametrization of convex bodies, namely the support function. In that case, we get the regularity of the support function, which does not imply the regularity of the corresponding shape, but only 
the fact that this one is strictly convex.

\subsection{Proof of Theorem \ref{th:reg} and \ref{th:reg+constraint}}\label{sect:proof-reg}

\subsubsection*{First order optimality condition:}

A first optimality condition for the problem \eqref{eq:min} is stated in \cite[Proposition 3.1, 3.2]{LN} when $j$ is defined and differentiable in the Sobolev Hilbert space $H^1(\T)$. We give here an adaptation to state this result in $W^{1,\infty}$ instead (which is important for our applications involving a PDE, since the shape functionals are known to be differentiable for Lipschitz deformations only).
\begin{proposition}\label{prop:order1}
Let $u_0>0$ be a solution of (\ref{eq:min}) with $j:W^{1,\infty}(\T)\to\R$ of class $ C^1$ and such that $j'(u_{0})\in  C^0(\T)'$. Then there exists $\zeta_0 \in W^{1,\infty}(\T)$, such that 
\begin{equation}\label{eq:euler1}
\left\{\begin{array}{l}
\zeta_0\geq0\;\textrm{on}\;\T, \quad \zeta_0=0\;\textrm{on}\;\Supp(u_{0}''+u_{0}), \quad\mbox{ and }\\[2mm]
\forall\; v\in W^{1,\infty}(\T_{in}),\quad
j'(u_0)v=
\displaystyle{
\langle\zeta_0+\zeta_0'',v\rangle_{(W^{1,\infty})'\times W^{1,\infty}}}:=\int_{\T}\zeta_{0}v-\zeta_{0}'v'.
\end{array}\right.\end{equation}
\end{proposition}

\begin{remark}
Without any assumption on $j'(u_{0})$, we would a priori get a Lagrange multiplier $\zeta_{0}\in L^\infty(\T)$ (see the proof below). The non-continuity of $\zeta_{0}$ may lead to some difficulties, especially to state that $\zeta_{0}=0$ on $\Supp(u''_{0}+u_{0})$. Though a restriction, the assumption $j'(u_{0})\in C^0(\T)'$ will be satisfied in all of our applications.
\end{remark}

\begin{proof}
We set
$$g:v\in W^{1,\infty}\mapsto v''+v \in (W^{1,\infty})'\textrm{ in the sense that }\langle v''+v,\varphi\rangle_{{W^{1,\infty}}'\times W^{1,\infty}}=\int_{\T}v\varphi-v'\varphi',$$
and we consider 
$Y:={\rm Im}(g)=\{f\in W^{1,\infty}(\T)',\; 
\langle f, \cos\rangle_{(W^{1,\infty})'\times W^{1,\infty}}=
\langle f, \sin\rangle_{(W^{1,\infty})'\times W^{1,\infty}}=0\}$, which is a closed subspace of 
$(W^{1,\infty}(\T))'$.

Applying the same strategy as in \cite{LN}, one gets $l_{0}\in Y'$ such that
$l_0(g(u_0))=0$ and
$$
\forall f\in Y,\; f\geq0\Rightarrow l_{0}(f)\geq 0,\quad
\textrm{ and }\quad\forall v\in {W^{1,\infty}},\;
\langle j'(u_{0}),v\rangle_{(W^{1,\infty})'\times {W^{1,\infty}}}=\langle l_{0},v''+v\rangle_{Y'\times Y}.$$

We restrict ourselves to $v\in \mathcal{D}(\T):= C^\infty(\T)$, and consider 
$$\zeta_{0}:f\in\D(\T)\cap Y\mapsto \langle \zeta_{0}, f\rangle_{\D'\times\D}:=\langle l_{0}, f\rangle_{Y'\times Y}.
$$
Our aim is to prove that $\zeta_{0}$ can be extended to a continuous linear form on $L^1(\T)$. 
First, for $f\in \D(\T)\cap Y=\{f\in \D(\T),\; \int_\T f\sin = \int_T f\cos =0\}$
we choose the unique $v\in W^{2,1}(\T)$ such that $\{\int_\T v\sin = \int_\T v\cos =0\}$ and 
$v''+v=f$ in $\T$.
Then there exists $C<\infty$ independant of $v$ or $f$ such that 
\begin{equation}\label{eq:estimate}
\|v\|_{W^{1,\infty}(\T)}\leq C\|f\|_{L^1(\T)}.
\end{equation}

Indeed, we first get an $L^\infty$-estimate using Fourier series: if $f=\sum_{n\in\Z}\widehat{f}(n)e_n$ with $e_{n}(\theta)=e^{in\theta}$ and $\widehat{f}(n)=\int_{\T}f(\theta)e^{-in\theta}\frac{d\theta}{2\pi}$,  then $v=\sum_{|n|\neq 1}\frac{1}{1-n^2}\widehat{f}(n)e_{n}$, and therefore
$$\|v\|_{L^\infty}\leq\left(\sum_{|n|\neq 1}\frac{1}{|1-n^2|}\right)\max_{n} |\widehat{f}(n)|\leq C\|f\|_{L^1},$$
with $C<\infty$. Then we get a $W^{1,\infty}$-estimate by choosing $\theta_{0}$ such that $v'(\theta_{0})=0$ (which is always possible, thanks to regularity and periodicity of $v$), and getting from $v''+v=f$ that
$$|v'(\theta)|=\left|-\int_{\theta_{0}}^\theta (f(s)-v(s))ds\right|\leq 2\pi\left(\|v\|_{L^\infty}+\|f\|_{L^1}\right),$$
which concludes the proof of the estimate \eqref{eq:estimate}.

Therefore, we can write ($C$ may define different universal constants)
\begin{equation}\label{eq:L1}
\forall f\in Y\cap\D(\T),\quad|\langle \zeta_{0}, f\rangle_{\D'\times\D}|=|\langle l, v''+v\rangle_{Y'\times Y}|=|\langle j'(u_{0}),v\rangle_{(W^{1,\infty})'\times {W^{1,\infty}}}|\leq C\|v\|_{W^{1,\infty}}\leq C\|f\|_{L^1}.
\end{equation}
We now extend $\zeta_{0}$ on $\D(\T)$ by 
$$\forall f\in \D(\T), \quad\langle\zeta_{0},f\rangle_{\D'\times\D}=\langle\zeta_{0},f-\widehat{f}(1)e_1-\widehat{f}(-1)e_{-1}\rangle_{\D'\times\D}.$$
Then, applying \eqref{eq:L1} to $f-\widehat{f}(1)e_1-\widehat{f}(-1)e_{-1}$, we get
$$\forall f\in\D(\T),\quad |\langle \zeta_{0}, f\rangle_{\D'\times\D}|\leq C\|f-\widehat{f}(1)e_1-\widehat{f}(-1)e_{-1}\|_{L^1}\leq C\|f\|_{L^1},$$
and therefore by density, we extend $\zeta_{0}$ to a  continuous linear form in $L^1$, which can be identified with $\zeta_{0}\in L^\infty$.
Moreover, in the sense of distributions:
$$\langle \zeta_{0}, v''+v\rangle_{\D'\times\D}=\langle j'(u_{0}),v\rangle_{\D'\times \D}\textrm{, that is to say }\zeta_{0}''+\zeta_{0}=j'(u_{0}).$$
From the hypothesis for $j'(u_0)$ it follows $\zeta_{0}''+\zeta_{0}\in ( C^0(\T))'$ which implies $\zeta_{0} \in W^{1,\infty}(\T).$
Using the continuity of $\zeta_{0}$ and the fact $j'(u_0)(u_0)=0$ we get $\int_{\T} \zeta_{0} d(u_{0}''+u_{0})=0$ by a density argument. Therefore,
the rest of the proof stays as in \cite{LN}, namely we prove that we can add a combination of $\cos$ and $\sin$ to $\zeta_{0}$ so that $\zeta_{0}\geq 0$.
\end{proof}

\subsubsection*{Proof of Theorem \ref{th:reg}.}

Applying the previous proposition, and using the hypotheses on the functional $j$, we get:
\begin{eqnarray*}
\forall v\in  C^\infty_{c}(\T_{in}),\;\; j'(u_0)v
=
r'(u)v+\int_{\T}G_u(\theta,u_0,u_0')v+G_q(\theta,u_0,u_0')v'
&=&
\langle\zeta_0+\zeta_0'',v\rangle_{(W^{1,\infty}(\T))'\times W^{1,\infty}(\T)}.
\end{eqnarray*}
To integrate by part in this formula, since $u_{0}'$ is only in $BV(\T)$, we may look in \cite{V} (see also \cite{ambrosio}) to get:
\begin{equation}\label{eq:opti1}
r'(u_{0})+G_u(\theta,u_0,u_0')-
G_{\theta q}(\theta,u_0,u_0')-G_{uq}(\theta,u_0,u_0')u_0' -
u_{0}''\widetilde{G_{qq}}(\theta,u_0,u_0')=\zeta_0+\zeta_0'' \textrm{ in }\mathcal{D}'(\T_{in}).
\end{equation}
where $\widetilde{G_{qq}}(\theta,u_0,u_0')=\int_0^1G_{qq}(\theta,u_0(\theta),(1-t)u_0'(\theta^+) + tu_0'(\theta^-))dt$. For simplicity, we will drop the indication of the dependence in $(\theta,u_0,u_0')$ and write more simply 
\begin{equation}\label{eq:opti2}
r'(u_{0})+G_u-
G_{\theta q}-G_{uq}u_0' -
u_{0}''\widetilde{G_{qq}}=\zeta_0+\zeta_0'' \textrm{ in }\mathcal{D}'(\T_{in}).
\end{equation}
Equality \eqref{eq:opti2} implies that $\zeta_{0}''$ is a Radon measure, and also that the singular parts of the measures in the two sides of \eqref{eq:opti2} are equal. To study the sign of these measures, we will use the following lemma.

\begin{lemma}\label{l:zeta''} 
The measure $\zeta_0''$ satisfies:
$\zeta_0''\geq 0$ on $[\zeta_0=0]$.
\end{lemma}

\noindent
{\bf Proof of Lemma \ref{l:zeta''}.} Let $\varphi \in {\cal C}_0^\infty(\R), \varphi\geq 0$ and let $p_n:\R_{+}\to \R_{+}$ be defined by
$$\forall r\in [0,1/n],\;p_n(r)= 1-nr;\;\forall r\in [1/n,+\infty),\;p_n(r)=0.$$
Recall that $\zeta_0\in W^{1,\infty}(\T)$ and $\zeta_0\geq 0$. Then
$$\int\varphi p_n(\zeta_0)d(\zeta_0'')=-\left(\int \zeta_0'\varphi' p_n(\zeta_0)+\varphi p_n'(\zeta_0){\zeta_0'}^2\right)\geq -\int\zeta_0'\varphi'p_n(\zeta_0).$$
Letting $n$ tend to $+\infty$ leads to
$$\int_{[\zeta_0=0]}\varphi d(\zeta_0'')\geq -\int_{[\zeta_0=0]}\zeta_0' \varphi'=0,$$
the last integral being equal to $0$ thanks to the known property $\zeta_0'=0$ a.e. on $[\zeta_0=0]$.
\hfill\qed

\noindent
{\bf End of the proof of Theorem \ref{th:reg}:}\\[3mm]
Denote $K:=\Supp(u_0''+u_0)$. Recall that $\zeta_0=0$ on $K$ by Proposition \ref{prop:order1}. By Lemma \ref{l:zeta''}, $\zeta_0''\geq 0$ on $K$. Let $u_{0}''=\mu_{ac}+\mu_{s}$ and ${\zeta_{0}''}=n_{ac}+n_{s}$ be the Radon-Nikodym decompositions of the measures $u_0'', \zeta_0''$ in their absolutely
continuous and singular parts. Note that:  [$u_0''+u_0\geq 0\Rightarrow \mu_s\geq 0$] and $n_s\geq 0$ on $K$.  

Identifying the singular parts in the identity (\ref{eq:opti2}), and using that $r'(u_0), G_u,  G_{\theta q}, G_{uq}u'_0, u_0\widetilde{G_{qq}}$ are at least $L^p$-functions, we are led  to $-\mu_s\widetilde{G_{qq}}=n_s$ in $\T_{in}$.
Since $\widetilde{G_{qq}}>0, \mu_s\geq 0, n_s\geq 0$ on $K\supset \Supp(\mu_s)$, we deduce $\mu_{s}=0=n_s$ in $\T_{in}$. Thus, $u_0\in W^{2,1}(\T_{in})$ and $u'_0$ is absolutely continuous on $\T_{in}$. In particular, $\widetilde{G_{qq}}=G_{qq}$ on $\T_{in}$.

We can now obtain higher regularity, using again the multiplier $\zeta_{0}''$.
Indeed, 
on one hand, we deduce from Lemma \ref{l:zeta''}, from (\ref{eq:opti2}) and from the inequality $-u_0''G_{qq}\leq u_0G_{qq}$, that, on the set $\T_{in}\cap K$
$$
0
\leq 
\zeta_{0}''
\leq 
r'(u_{0})+G_u-G_{\theta q}-G_{uq}u_0' +u_{0}G_{qq} \in L^p(\T).
$$
Thus, $\zeta_0''\in L^p(\T_{in}\cap K)$. Going back to (\ref{eq:opti2}) and using that $\widetilde{G_{qq}}=G_{qq}$ is bounded from below on the compact set $\T\times u_0(\T)\times Conv(u'_0(\T))$, we deduce $u_0''\in L^p(\T_{in}\cap K)$.

On the other hand, 
in the open set $\T_{in}\setminus K$, we have $u_0''+u_0=0$ so that $u_{0}''\in L^\infty(\T_{in}\setminus K)$. As a conclusion $u_{0}''\in L^p(\T_{in})$.
\qed

\subsubsection*{Proof of Theorem \ref{th:reg+constraint}.}
Optimality conditions are written with the Lagrangian (since $m'(u_{0})$ is onto, see also \cite[Proposition 2.3.3]{LN}):
$$
\forall v\in C^\infty_{c}(\T_{in}),\quad j'(u_0)v+\mu\cdot (m'(u_{0})v)
=
\displaystyle{
\langle\zeta_0+\zeta_0'',v\rangle_{(W^{1,\infty})'\times W^{1,\infty}}},$$
for some $\mu\in\R^d$.
The regularity of $m'(u_{0})$ implies that the strategy used in the proof of Theorem \ref{th:reg} remains valid.
\hfill\qed

\subsection{Examples}\label{sect:examples-smooth}
In this section, we apply Corollary \ref{cor:reg} to a number of classical energy functionals. For the proof of 
the differentiability of the shape functionals see Section  \ref{sect:derivative1}. 
We start by reminding some classical PDE functionals that we use in our examples.

\subsubsection*{Dirichlet energy - Torsional rigidity}

For $\Om$ an open bounded set in $\R^2$, we consider the solution of the following PDE, in a variational sense:
\begin{equation}\label{eq:dirichlet}
U_\Om\in H^1_0(\Om),\quad -\Delta U_\Om =f\;\;\textrm{in}\;\Om,
\end{equation}
and we define the Dirichlet energy of $\Om$ by
\begin{eqnarray*}
E_f(\Om)
&:=&
\int_\Om\left(\frac{1}{2}|\nabla U_\Om|^2-f\,U_\Om\right)=\min\left\{\int_{\Omega}\left(\frac{1}{2}|\nabla U|^2-fU\right),\; U\in H^1_0(\Omega)\right\}	\\
&=&
-\frac{1}{2}\int_\Om|\nabla U_\Om|^2=-\frac{1}{2}\int_\Omega U_\Omega f.
\end{eqnarray*}
About the regularity of the state function, we are going to use the following classical result (see \cite{kadlec-1},
\cite{grisvard-1}).
\begin{lemma}\label{l:U->H2}
Let $\Om$ be convex, $f\in L^p_{loc}(\mathbb R^2)$ with $p>2$, and $U_\Om$ be the solution of (\ref{eq:dirichlet}). 
Then $U_\Om\in W^{1,\infty}(\Omega)\cap H^2(\Omega)$.
\end{lemma}

\begin{remark}
When $f\equiv 1$, the Dirichlet energy is linked to the so-called torsional rigidity $T(\Om)$, with the formula $T(\Om)=-2E_{1}(\Om)$.
\end{remark}

\subsubsection*{First Dirichlet-eigenvalue of the Laplace operator}

We define $\lambda_1(\Omega)$ as the first eigenvalue for the Laplacian with Dirichlet's boundary conditions on $\partial\Om$. It is well-known that, if we define $U_{\Omega}$ as a solution
of the following minimization problem,
\begin{equation*}
\lambda_1(\Omega):=\int_{\Omega} |\nabla U_{\Omega}|^2
=
\min\left\{\int_{\Omega}|\nabla U|^2,\;  U\in H^1_0(\Omega),\; \int_{\Omega}U^2=1 \right\},
\end{equation*}
then $U_{\Omega}$ is (up to the sign) the positive first eigenfunction of $-\Delta$ in $\Om$:
\[
U_{\Omega} \in H^1_0(\Omega),\quad
-\Delta U_{\Omega}  =  \lambda_1(\Omega)U_{\Omega},\quad
\int_\Omega U_{\Om}^2=1.
\]
Again, like in Lemma \ref{l:U->H2}, if $\Omega$ is convex then $U_{\Om}\in H^2(\Om)\cap W^{1,\infty}(\Om)$ and $U_\Omega>0$ in $\Omega$.\\

We are now in position to state some applications of Corollary \ref{cor:reg}:
\begin{example}[Penalization by perimeter]
One can study 
\begin{equation}\label{eq:ex1}
\min \{J(\Om):=F\left(|\Om|,E_{f}(\Om),\la_{1}(\Om)\right)+P(\Om)\;/\;\Om\textrm{ convex}, D_{1}\subset\Om\subset D_{2}\}
\end{equation}
where $F:(0,+\infty)\times(-\infty,0)\times(0,\infty)\to\R$ is $ C^1$, $f\in H^1_{loc}(\R^2)$, $D_{1},D_{2}$ are bounded open sets, $E_f(\Omega)$ is the Dirichlet energy and $\lambda_1(\Omega)$ is the first eigenvalue of $-\Delta$ defined as above.

\begin{proposition}\label{prop:ex1}
If $\Om_0$ is an optimal set for the problem \eqref{eq:ex1}, then 
the free boundary $\partial\Om_0\cap (D_{2}\setminus \overline{D_{1}})$ is $ C^{1,1}$ (or equivalently $W^{2,\infty}$), that is to say $\partial\Om_0\cap (D_{2}\setminus \overline{D_{1}})$ has a bounded curvature.
\end{proposition}
The proof is a simple consequence of Section \ref{sect:derivative1}, which asserts that $R(\Om)=F(|\Om|,E_{f}(\Om),\la_{1}(\Om))$ and $C(\Om)=P(\Om)$ satisfy the assumptions in Corollary \ref{cor:reg} with $p=\infty$.

Note that in Proposition \ref{prop:ex1}  we could also add a dependence of $F$ in the capacity of $\Om$ or in any shape functional which is shape differentiable and whose shape derivative can be represented as a function of $L^\infty(\partial\Om)$ when $\Om$ is convex.
\begin{remark}\label{rk:constraint}
The constraints $D_{1}\subset\Om\subset D_{2}$ helps existence for the problem \eqref{eq:ex1}. Of course, if one can prove existence of an optimal shape without these constraints (mainly, one need to prove that a minimizing sequence remains bounded and does not converge to a segment), the result of Proposition \ref{prop:ex1} remains a fortiori true for the whole boundary of the optimal shape,
i.e. $\partial\Omega_0$ is $ C^{1,1}$.
\end{remark}
\end{example}

\begin{example}[Volume constraint and Perimeter penalization]
We can also consider a similar problem with a volume constraint:
\begin{equation*}\label{eq:ex2}
\min \{ J(\Om):=F(E_{f}(\Om), \la_{1}(\Om)) + P(\Om)\;/ \; \Om\textrm{ convex}, \textrm{ and }|\Om|=V_{0}\},
\quad
V_{0}\in(0,+\infty).
\end{equation*}
In this case, the first optimality condition will be similar to the one for the problem \eqref{eq:ex1} with
${F(E_{f}(\Om),\la_{1}(\Om))+\mu |\Omega|+P(\Om)}$, where $\mu$ is a Lagrange multiplier for the constraint $|\Omega|=V_{0}$. Theorem \ref{th:reg+constraint} applies and one gets globally the same regularity result (but global) as in Proposition \ref{prop:ex1} on any optimal shape.
\end{example}

\begin{example}[Perimeter constraint]\label{ex:perimeter}
If one considers again a problem with a perimeter constraint, 
\begin{equation}\label{eq:ex3}
\min \{J(\Om):=F(|\Om|,E_{f}(\Om),\la_{1}(\Om))\;/\;\Om\textrm{ convex}, \textrm{ and }P(\Om)=P_{0}\}
\end{equation}
where $P_{0}\in(0,+\infty)$, one needs to be more careful. In this case, the first optimality condition will be similar to the one for the problem \eqref{eq:ex1}, with $F(|\Om|,E_{f}(\Om),\la_{1}(\Om))+\mu P(\Omega)$, where $\mu$ is a Lagrange multiplier for the constraint $P(\Omega)=P_{0}$. 
Therefore if we are able to prove $\mu>0$ then we can apply the same strategy as in Theorem \ref{th:reg}, and we therefore get the same regularity result as in Proposition \ref{prop:ex1}. However, if $\mu<0$, we refer to Example \ref{ex:ex'3}.
\end{example}

\begin{example}
In a more abstract context, one can consider
\begin{equation}\label{eq:ex4}
\min \{J(\Om)-\alpha|\Om|+P(\Om)\;/\;\Om\textrm{ convex }\subset D\},
\end{equation}
where $J$ is a shape differentiable functional, increasing with respect to the domain inclusion, $D$ is an open set, and $\alpha>0$ (if $\alpha=0$, the empty set is clearly solution of the problem). Again, we get that $\partial\Om_0\cap D$ has a locally bounded curvature. Indeed, the derivative of $j(u):=J(\Om_{u})$ is a nonpositive measure, thanks to the monotonicity of $J$ (see \cite{LP}), and we apply Theorem \ref{th:reg} combined with the end of Remark \ref{rk:measure}.
\end{example}

\subsection
{Computation and estimate of first order shape derivatives}\label{sect:derivative1}

In this section we will prove the differentiability of the shape functionals involved in the examples of Section 
\ref{sect:examples-smooth}, which are needed in Proposition \ref{prop:ex1}.

\subsubsection{Volume and perimeter}\label{ssect:volper}
About geometrical functionals, it is easy to write the area and the perimeter as functional of $u$, namely

\begin{equation}\label{e:a,p}
a(u):=|\Om_{u}|=\int_{\T} \frac{1}{2u^2}d\theta, \;\; \qquad 
p(u):=P(\Om_{u})=\int_{\T} \frac{\sqrt{u^2+u'^2}}{u^2}d\theta,\quad
u\in W^{1,\infty}(\T)\cap\{u>0\}.
\end{equation}
Note that 
$p(u)=\int_{\T}G(\theta,u(\theta),u'(\theta))d\theta$ with $G(\theta,u,q)=\frac{\sqrt{u^2+q^2}}{u^2}$ and one can easily check that 
$G_{qq}=\frac{1}{(u^2+q^2)^{3/2}}>0$.

\subsubsection{Dirichlet Energy - Torsional rigidity}\label{sect:dir}
We focus our analysis around a convex open set $\Omega_0$ with parametrization $u_0>0$. For $\|u-u_0\|_{W^{1,\infty}(\T)}$ small, consider
\[
\begin{array}{cccl}
e_f:&W^{1,\infty}(\T)\cap\{u>0\}&\to&\R,\\
&u&\mapsto&E_f(\Om_u).
\end{array}
\]

In order to study the differentiability of $e_f$ near $u_0$, we use the classical framework of shape derivatives.
As usual, we need to work with an extension operator: the deformation $\partial\Omega_0$ to $\partial\Omega_u$ allows to define the vector field
$\xi(u):\partial\Omega_0\to\R^2$ such that 
$\partial\Omega_{u}=(Id+\xi(u))(\partial\Omega_0)$.
We will consider an extension to $\R^2$ of this transformation, since we need to study the differentiability of 
$u\to \hat{U}_u:=U_{\Omega_u}\circ(Id+\xi(u))\in H^1_0(\Omega_0)$, where 
$U_u:=U_{\Omega_u}$ (see \cite{HP} for example).

If we consider a smooth extension operator
$
\xi: W^{1,\infty}(\T) \to W^{1,\infty}(\R^2;\R^2),
$
we have $(Id+\xi(u))(\partial\Om_{0})=\partial\Om_u$ if 
\begin{equation}\label{eq:xi-bord}
\xi(u)\left(\frac{1}{u_0(\theta)},\theta\right)
=
\left(\frac{1}{u(\theta)}-\frac{1}{u_0(\theta)}\right)e^{i\theta}, 
\forall \theta \in \T,
\end{equation}
where $(\frac{1}{u_{0}}(\theta),\theta)$ are polar coordinates (for simplicity, we will often write $u_0$, $u$ or $\xi$ instead of $u_0(\theta)$, $u(\theta)$ or $\xi(u)(r,\theta)$).

\begin{remark}\label{r:polarxi}
The transformation $\xi(u)$ can be extended to $\R^2$ in different ways. 
The easiest way is to take 
\begin{equation}\label{eq:xi-radial}
\xi(u)(r,\theta)=\left(\frac{1}{u(\theta)}-\frac{1}{u_{0}(\theta)}\right)e^{i\theta}\eta(r,\theta)\quad \textrm{ in }\;\; \R^2,
\end{equation}
where  $\eta\in C_0^\infty(\mathbb R^2)$, $\eta=0$ in a neighborhood of the origin and $\eta=1$ in a neighborhood of $\partial\Omega_0$.

This (polar) extension  of $\xi(u)$ is such that 
$\xi\in  C^\infty(W^{1,\infty}(\T); W^{1,\infty}(\R^2;\R^2))$ near $u_0$, and is sufficient for the results of this section. More work will be needed for the second order shape derivatives, see Section \ref{ssect:dir2}.

Let us point out that if $\xi$ is $ C^2$ in a neighborhood of $u_0$ and satisfies \eqref{eq:xi-bord},
then 
\begin{equation}\label{e:xi'(u),xi''(u)}
\forall v\in W^{1,\infty}(\T):\quad
\xi'(u_0)(v)
=
-\frac{v}{u^2_0}e^{i\theta},
\quad
\xi''(u_0)(v,v)
=
2\frac{v^2}{u_0^3}e^{i\theta}\quad \textrm{on}\;\; \partial\Om_{0}.
\end{equation}
Note also that the method used in the proof of Lemma \ref{l:hU->Ck}, which is needed in the proof of Proposition \ref{prop:dir}, allows to say that  the method a priori fails if we consider an extension operator $\xi: H^{1}(\T)\to H^1(\R^2;\R^2)$. This explains our choice to work with $v\in W^{1,\infty}(\T)$ rather than $v\in H^1(\T)$, even though it introduces extra difficulties (like in Proposition \ref{prop:order1} and in the proof of Proposition \ref{prop:dir2}). 
\hfill$\Box$
\end{remark}

\noindent
The main result of this section is the following.
\begin{proposition}\label{prop:dir}
Let $\Om_{0}=\Om_{u_{0}}$ convex, $f\in H^k_{loc}(\R^2)$, $k\in\mathbb N^*$
and $\xi\in C^k(W^{1,\infty}(\T);W^{1,.\infty}(\R^2;\R^2))$ near $u_0$. 
We have:\\
i) 
$e_f$ is $ C^k$ near $u_0$.\\
ii)
If $\xi$ satisfies \eqref{eq:xi-bord}, then for any $v\in W^{1,\infty}(\T)$ we have
\begin{eqnarray}
e_f'(u_0)(v) 
&=&
-\int_{\partial\Omega_0}\frac{1}{2} |\nabla U_0|^2(\xi'(u_0)(v)\cdot\nu_0)ds_0 		
=
\int_{\T}\frac{1}{2} |\nabla U_0(x_{\theta})|^2 \frac{v(\theta)}{u^3_0(\theta)}d\theta,		\label{e:e'(u)(v)-2}
\end{eqnarray}
where $U_0\in H^2(\Om_0)$ is the solution of \eqref{eq:dirichlet} in $\Om_0$, 
$\nu_0$ is the exterior unit normal vector on $\partial\Omega_0$,
$x_{\theta}=\frac{1}{u_0(\theta)}(\cos\theta,\sin\theta)\in\partial\Om_0$.
\\
iii) Furthermore, $e_{f}'(u_{0})\in L^\infty(\T)$.
\end{proposition}

\noindent
The proof of this proposition is classical and uses the following lemma,
which will be needed in the following section.
\begin{lemma}\label{l:hU->Ck}
Let $u_0\in W^{1,\infty}(\T)$, $u_0>0$, $f\in H^k_{loc}(\mathbb R^2)$, $k\in\mathbb N^*$.
We have:
\begin{enumerate}
\item[i)] 
The map $u\in W^{1,\infty}(\T)\mapsto \hat{U}_u\in H^1(\Omega_0)$ is $ C^k$ near $u_0$.
\item[ii)]
For $v\in W^{1,\infty}(\T)$,  set 
\begin{eqnarray}
\hat{U}_0':=\hat{U}_u'(u_0)(v),\quad
U'_0
&:=&
\hat{U}_0'-\nabla U_0\cdot\xi'(u_0)(v). \label{e:Ut'}
\end{eqnarray}
Then
\begin{eqnarray}
U_0'\in L^2(\Omega_0),\quad
\Delta U_0'
&=&
0\;\textrm{ in }\mathcal D'(\Omega_0),								\label{e:Delta Ut'}\\
U_0'+\nabla U_0\cdot\xi'(u_0)(v) &\in&  H^1_0(\Omega_0).	 		 			\label{e:bcUt'}
\end{eqnarray}
\item[iii)] 
Furthermore, if $u_{0}''+u_{0}\geq0$, then $U'_0\in H^1(\Omega_0)$.
\end{enumerate}
\end{lemma}
\begin{remark}
Here we are not interested in the differentiability of $u\mapsto U_u$
and the function $U'_0$ is directly defined by \eqref{e:Ut'}. 
In fact,  the map $u\mapsto U_u$ (with $U_u$ extended by zero in $\R^2$) is differentiable in $L^2(\R^2)$
and its derivative equals $U_0'$ in $\Omega_0$, see  Th\'eor\`eme 5.3.1, \cite{HP}
for example.
\end{remark}

\noindent
{\bf Proof of Lemma \ref{l:hU->Ck}:}
\begin{enumerate}
\item[i)] 
The map  
$\theta\in W^{1,\infty}(\R^2;\R^2)\mapsto U_{(Id+\theta)(\Om_{0})}\circ (Id+\theta)\in H^1_{0}(\Om_{0})$ is $ C^k$ in a neighborhood of $0$, 
see for example \cite[Proposition 5.3.7]{HP}.
We conclude by using the composition of this map with $\xi$.

\item[ii)] 
It is clear that $U_0'\in L^2(\Omega_0)$ and that 
$U_0'+\nabla U_0\cdot\xi'(u_0)(v)=\hat{U}_0'\in H^1_0(\Om_0)$.
To prove $\Delta U_0'=0$ we consider the map 
$S:W^{1,\infty}(\T)\mapsto W^{1,\infty}(\R^2;\R^2)$,
$S(u)=(Id+\xi(u))^{-1}$, which is well defined and $ C^k$ in a neighborhood of $u_0$. From $S(u)\circ(Id+\xi(u))=Id$, it is easy to check that for
$v\in W^{1,\infty}(\T)$ we have
\begin{equation}\label{e:S0',S0''}
 S'(u_0)((v)=-\xi'(u_0)(v),\qquad
 S''(u_0)(v,v)=2\nabla\xi'(u_0)(v)\cdot\xi'(u_0)(v)-\xi''(u_0)(v,v).
\end{equation}
Let $\varphi\in\mathcal D(\Omega_0)$. From \eqref{eq:dirichlet}, for all $u$ near $u_0$  we have
$
 \int_{\Om_0}\hat{U}_u\circ S(u)\Delta\varphi - f\varphi=0.
$
Differentiating this equality on the direction $v$ gives
\begin{equation}\label{e:inthuLapphi'}
 \int_{\Om_0}
 \left(
 \hat{U}_u'\circ S(u) + 
 \nabla\hat{U}_u\circ S(u)\cdot S'(u_0)(v)
 \right)\Delta\varphi=0.
\end{equation}
Replacing $u=u_0$ in (\ref{e:inthuLapphi'}) and using \eqref{e:S0',S0''} gives
\[
 \int_{\Om_0}
 \left(\hat{U}'_0-\nabla U_0\cdot\xi'(u_0)(v)\right)\Delta\varphi=0,
\]
which proves ii).
\item[iii)]
If $u_{0}''+u_{0}\geq0$ then $\Omega_0$ is convex. From Lemma \ref{l:U->H2} we obtain
$U_0\in H^2(\Omega_0)$, which implies $U_0'\in H^1(\Om_0)$.
\hfill$\Box$
\end{enumerate}

\noindent\hfill\\
{\bf Proof of Proposition \ref{prop:dir}:}\\
i) 
The functional
$u\mapsto e_{f}(u)$
can be seen as $e_f(u)=\mathcal E_{f,\Omega_0}\circ\xi(u)$, where
$\mathcal E_{f,\Omega_0}$ is a classical functional, introduced to compute shape derivatives:
\begin{equation}\label{e:calE}
\begin{array}{cccl}
\mathcal{E}_{f,\Om_{0}}(\theta):&W^{1,\infty}(\R^2;\R^2)&\to&\R\\
&\theta&\mapsto&E_f((Id+\theta)(\Om_0)).
\end{array}
\end{equation}
As $\xi$ is $ C^k$ near $u_0$ and $\mathcal E_{f,\Om_0}$ is $ C^k$ near $\theta=0$ in 
$W^{1,\infty}(\mathbb R^2;\R^2)$, see \cite[Corollaire 5.3.8]{HP}),
the differentiability of $e_f(u)$ follows.\\
ii)
As we have
$e_f(u)=-\frac{1}{2}\int_{\Omega_u} \hat{U}_u\circ S(u) f$ and
$\hat{U}_u=0$ on $\partial\Omega_0$, from Corollaire 5.2.5, \cite{HP}, we obtain
\begin{equation}\label{e:ef'(u)(v)}
e_f'(u)(v)
=
-\frac{1}{2}
\int_{\Omega_u}
\left(
\hat{U}_u'\circ S(u)+\nabla\hat{U}_u\circ S(u)\cdot S'(u)(v)
\right)f.
\end{equation}
Taking $u=u_0$ in the last equality and using (\ref{e:S0',S0''}) gives
\[
e_f'(u_0)(v)
=
-\frac{1}{2}
\int_{\Omega_0}
(\hat{U}_0' - \nabla U_0\cdot\xi'(u_0)(v))f 
=
-\frac{1}{2}\int_{\Omega_0}
U_0'f
=
-\frac{1}{2}
\int_{\partial\Omega_0}|\nabla U_0|^2(\xi'(u_0)(v)\cdot\nu_0)ds_0.
\]
Finally, by changing the variable 
$s_0=\frac{\sqrt{u^2_0+(u'_0)^2}}{u^2_0}d\theta$, taking into account that
$\nu_0=\left(\frac{1}{u_0}e^{i\theta}+\frac{u'_0}{u^2_0}(ie^{i\theta})\right)\frac{u^2_0}{\sqrt{u^2_0+u'^2_0}}$,
and after using \eqref{e:xi'(u),xi''(u)}.
we obtain (\ref{e:e'(u)(v)-2}).\\
iii)
As $k\in N^*$ it follows $f\in L^p(\Omega_0)$, for all $p\in[1,\infty)$.
Then Lemma \ref{l:U->H2} gives $U_0\in W^{1,\infty}(\Omega_0)$, so $e_f'(u_0)\in L^\infty(\T)$.
\hfill$\Box$

\subsubsection{First eigenvalue of the Laplace operator with Dirichlet boundary conditions}\label{sect:la1}
We consider
$$\begin{array}{cccl}
l_{1}:&\{u\in W^{1,\infty}(\T),\; u>0\}&\to&\R\\
&u&\mapsto&l_1(u):=\lambda_{1}(\Om_{u})
\end{array}
$$
and we have the same result as in Proposition \ref{prop:dir}, see for example Th\'eor\`eme 5.7.1, \cite{HP}, and \eqref{e:xi'(u),xi''(u)},
with
\[
l_{1}'(u_0)(v) = \int_{\T} |\nabla U_0|^2(x_{\theta}) \frac{v(\theta)}{u^3_0(\theta)}d\theta,
\quad
\forall v\in W^{1,\infty}(\T).
\]

\subsection{Application with the dual parametrization}\label{sect:support}
Instead of using parametrization by the gauge function, one can also use the well-known parametrization by the support function of a body, namely
$$
\forall\theta\in\T,\;\;  h_{\Om}(\theta):=\max\{x\cdot e^{i\theta},\; x\in\Omega\}.
$$
We get a characterization of the convexity in a similar way to \eqref{eq:conv}:
$$\Om\textrm{ is convex }\Rightarrow h_{\Om}''+h_{\Om}\geq 0.$$
Conversely, if $h\in W^{1,\infty}(\T)$ satisfies $h''+h\geq 0$, then one can find a unique (after a choice of an origin) open convex set, denoted $\Om^{h}$, whose support function is $h$ (see \cite{Schneider} for example). 
This parametrization is the dual of the one with the gauge function. Indeed, the gauge function of $\Om$ is the support function of the dual body of $\Om$ and vice versa.\\

Therefore the optimization problem
\begin{equation}\label{eq:minshapeh}
\min\{J(\Om)\;\;/\;\Om\in\mathcal{S}_{ad},\; \Omega\textrm{ convex}\},
\end{equation}
where $\mathcal{S}_{ad}$ is a class of open planar sets, becomes
\begin{equation}\label{eq:minh}
\left\{\begin{array}{l} 
\textrm{find }h_0\in \widetilde{\mathcal{F}_{ad}}\textrm{ such that }\widetilde{j}(h_0)=\min\{\widetilde{j}(h),\ h\in\widetilde{\mathcal{F}_{ad}},\; h''+h\geq 0\},\quad \textrm{ where }
 \\[3mm]
\widetilde{j}(h)=J(\Om^{h}), \textrm{ and }\widetilde{\mathcal{F}_{ad}}=\{h\in W^{1,\infty}(\T)\;/\;\Om^{h}\in\mathcal{S}_{ad}\},
\end{array}\right.
\end{equation}
which is the same as \eqref{eq:min}.\\

Again, if the set of admissible functions can be written
\begin{equation}\label{eq:Fadtilde}
\widetilde{\mathcal{F}_{ad}}=\{h\in W^{1,\infty}(\T)\;/\;k_{1}\leq h\leq k_{2}\},
\end{equation}
we can define
$\widetilde{\T_{in}}=\{\theta\in\T\; /\;k_{1}(\theta)<h(\theta)<k_{2}(\theta)\}$, and then
$\widetilde{(\partial\Om)_{in}}=\{x\in\partial\Om\textrm{ s.t. }\exists\theta\in\widetilde{\T_{in}}, x\cdot e^{i\theta}=h(\theta)\}$, i.e.
the set of points of $\partial\Om$ whose supporting plane is orthogonal to $(\cos(\theta),\sin(\theta))$ with $\theta\in\widetilde{\T_{in}}$.\\

As in Example \ref{ex:inclusion}, if $\mathcal{S}_{ad}=\{\Om\;/\;K_{1}\subset\Om\subset K_{2}\}$, where $K_{1}$ and $K_{2}$ are two convex open sets, then \eqref{eq:Fadtilde} is satisfied with $k_{1}, k_{2}$ the supports functions of $K_{1}, K_{2}$, and in that case $\widetilde{(\partial\Om)_{in}}=\partial\Om\setminus(\partial K_{1}\cup \partial K_{2})$.

Therefore one gets a dual version of Corollary \ref{cor:reg} as follows.

\begin{corollary}\label{cor:h} 
Let $\Om_{0}=\Om^{h_{0}}$ be an optimal shape for the problem \eqref{eq:minshapeh} with $J=R+C$, and assume that, 
$$\forall h\in\widetilde{\mathcal{F}_{ad}},\quad R(\Om^{h})=r(h)\textrm{ and }C(\Om^{h})=\int_{\T}G(\theta,h(\theta),h'(\theta))d\theta$$
where $r$ and $G$ satisfy the assumptions of Theorem \ref{th:reg} for some $p\in[1,\infty]$. Then 
$$h_{0}\in W^{2,p}(\widetilde{\T_{in}}).$$
This implies in particular that $\widetilde{(\partial\Om_{0})_{in}}$ is strictly convex.
\end{corollary}

\begin{remark}
This parametrization is especially interesting when one has to deal with the perimeter because in this case $P(\Om^{h})=\int_{\T}hd\theta$. An example of a function $C(\Om^{h})$ satisfying the hypotheses of Corollary \ref{cor:h} is now the opposite of the area, since
$$|\Om^{h}|=\frac{1}{2}\int_{\T}(h^2-h'^2)d\theta.$$

However, it is not easy to work now with functionals coming from PDE. Indeed, it is well-known for example, that the derivative of $\la_{1}$ in terms of $h$ is not more regular than a measure on $\T$, see \cite{J1,J2}. We think that this can be explained by the fact that some solutions of problems like \eqref{eq:ex1} may not be strictly convex.

\end{remark}

\section{Optimization of concave non-local shape functionals}\label{sect:polygon}

In this section, we prove the results of Section \ref{sssect:poly}. The main proof relies on the analysis of the second order shape derivatives.
Next we apply these results to various energy functionals involving
the Dirichlet energy or the first eigenvalue of the Laplace-Dirichlet operator. Since the optimal shapes come with no a priori regularity except the convexity condition, one needs some delicate computations to check the required assumptions. This leads to rather sharp estimates on second derivatives which are interesting for themselves.

\subsection{Proof of Theorems \ref{th:conc} and \ref{th:conc+constraint}}\label{sect:proof-conc}

We first introduce the classical Sobolev semi-norms on $\T$. For $s\in\R_{+}$, we set:
$$|u|_{H^s(\T)}^2:=\sum_{n\in\Z}|n|^{2s}|\widehat{u}(n)|^2\textrm{ where }\widehat{u}(n):=\int_{\T}u(\theta)e^{-in\theta}\frac{d\theta}{2\pi}.$$
We also define $H^s(\T):=\{u\in L^2(\T)\textrm{ such that }|u|_{H^s(\T)}<+\infty\}$ and $\|u\|_{H^s(\T)}^2:=\|u\|^2_{L^2(\T)}+|u|^2_{H^s(\T)}$.\\

\subsubsection*{Proof of Theorem \ref{th:conc}.}
The main idea is to prove that for a deformation supported by a small set, the estimate \eqref{eq:coercivite} is a concavity estimate, and so it violates the second order optimality condition. This relies of the following Poincar\'e-type inequality:
\begin{lemma}\label{lem:poincare}
Let $s\in[0,1)$ and $ \eps\in(0,\pi)$. Then there exists a constant $C=C(s)$ independant on $\eps$ such that,
$$\forall u\in H^1(\T)\textrm{ such that }{\Supp}(u)\subset[0,\eps],\; \|u\|_{H^{s}(\T)}\leq C\eps^{1-s}|u|_{H^1(\T)}.$$
\end{lemma}
{\bf Proof of Lemma \ref{lem:poincare}}.
Let $u\in  C^\infty(\T)$ with ${\Supp}(u)\subset[0,\eps]$.
If we first assume that $s=0$, then we have the classical Poincar\'e inequality (with the optimal constant), proved using the fact that $|u|_{H^1(\T)}^2=\int_{\T}u'^2$, so
$$\|u\|_{L^2(\T)}\leq \frac{\eps}{\pi}|u|_{H^1(\T)}.$$
If one has now $s\in(0,1)$, one can proceed with an interpolation inequality, easily obtained by H\"older inequality:
$$|u|_{H^{s}(\T)}^2=\sum_{n\in\Z}|n|^{2s}|\widehat{u}(n)|^{2s}|\widehat{u}(n)|^{2(1-s)}\leq \left(\sum_{n\in\Z}|n|^{2}|\widehat{u}(n)|^2\right)^{s}\left(\sum_{n\in\Z}|\widehat{u}(n)|^2\right)^{1-s},$$
and so
$$|u|_{H^{s}(\T)}\leq |u|_{H^1(\T)}^{s}\|u\|_{L^2(\T)}^{1-s}\leq\frac{\eps^{1-s}}{\pi^{1-s}}|u|_{H^1(\T)}.$$
\qed
Let $K:=\Supp(u_0''+u_{0})$. Assume that, for a connected component $I$ of $\T_{in}$,  $K\cap I$ is infinite. Then, there exists $\theta_{0}\in\overline{I}$ an accumulation point of $K\cap I$. Without loss of generality we can assume
$\theta_0=0$ and also that there exists a decreasing sequence $(\eps_n)$ tending to $0$ such that
$K\cap(0,\eps_n)\subset I$ is infinite. Then,  we follow an idea of T. Lachand-Robert and M.A. Peletier as in \cite{LN} (see also \cite{LRP}). We can always find $0<\eps_n^i<\eps_n$, $i=1,\ldots,4$, increasing with respect to $i$, such that ${\Supp}(u_0''+u_{0})\cap(\eps_n^i,\eps_n^{i+1})\neq\emptyset$, $i=1,3$. We consider $v_{n,i}\in W^{1,\infty}(\T)$ solving
\[
 v_{n,i}''+v_{n,i}=\mathbbmss{1}_{(\eps_n^i,\eps_n^{i+1})}(u_0''+u_0),\quad v_{n,i}=0 \mbox{ in }
(0,\eps_n)^c,\; i=1,\ldots,3.
\]
Such $v_{n,i}$ exist since we avoid the spectrum of the Laplace operator with Dirichlet boundary conditions.
Next, we look for $\lambda_{n,i},\ i=1,3$ such that
${\displaystyle v_n=\sum_{i=1,3}\lambda_{n,i} v_{n,i}}$ satisfies
\[
v'_n(0^+)=v'_n(\eps_n^-)=0.
\]
The above derivatives exist since $v_{n,i}$ are regular near $0$ and $\eps_n$ in $(0,\eps_n)$.
We can always find such $\lambda_{n,i}$ so as they satisfy two linear equations.
It implies that $v_n''$ does not have any Dirac mass at $0$ and $\eps_n$.
It even implies that the support of $v_n$ is included in $[\eps_n^1,\eps_n^4]$. In particular, $v_n''+v_n=\varphi (u_0''+u_0)$ where $\varphi$ is bounded and with support in $[\eps_n^1,\eps_n^4]$. As ${{\Supp}}(u_{0})\cap(\eps_{n}^i,\eps_{n}^{i+1})\neq \emptyset$, we also have $v_n\neq 0$.

 Since ${\Supp}(v_n)\subset \T_{in}$ and  $v_n''+v_n=\varphi (u_0''+u_0)$, it follows that $u_{0}+tv_{n}$ is admissible for $|t|$ small enough (and $n$ fixed).
Consequently, since $j(u_0+tv_n)\geq j(u_0)$ for $|t|$ small, we have $j'(u_0)(v)=0$
and then by using the assumption (\ref{eq:coercivite}) and Lemma \ref{lem:poincare}, we get

\begin{eqnarray}\label{eq:estimate1}
0\leq
j''(u_0)(v_n,v_n)&\leq&-\alpha|v_{n}|^2_{H^1(\T)} +\gamma|v_{n}|_{H^{1}(\T)}\|v_{n}\|_{H^{s}(\T)} +\beta\|v_{n}\|^2_{H^{s}(\T)}\\
&\leq&
(-\alpha+C\gamma\eps_{n}^{1-s}+C^2\beta(\eps_{n})^{2(1-s)})|v_{n}|^2_{H^1(\mathbb T)}.		\label{eq:estimate2}
\end{eqnarray}
As $\eps_n$ tends to $0$, inequality (\ref{eq:estimate1}) becomes impossible and
proves that ${\Supp}(u_0''+u_{0})$ has no accumulation points in $\T_{in}$.
It follows that $u_0''+u_0$ is a finite sum of positive Dirac masses.
\qed

\begin{remark}\label{rk:corner}
More precisely, we can get an estimate of the number of corners in each connected component $I$ of $\T_{in}$: 
$$\#\{\Supp(u_{0}''+u_{0})\cap I\}\leq \frac{2|I|}{A}+2\;\textrm{ where }\;A^{1-s}:=\frac{-\gamma+\sqrt{\gamma^2+4\alpha\beta}}{2\beta C}$$
($C=\frac{1}{\pi^{1-s}}$ appears in Lemma \ref{lem:poincare}).
Indeed, let us consider three consecutive Dirac masses $\theta_{1},\theta_{2},\theta_{3}$ in $I$. Then
\begin{itemize}
\item if $\beta>0, \gamma\geq 0$, we have
\begin{equation}\label{eq:ecart}
\left(\theta_{3}-\theta_{1}\right)^{1-s}\geq\frac{-\gamma+\sqrt{\gamma^2+4\alpha\beta}}{2\beta C}.
\end{equation}
\item if $\beta=\gamma=0$, then we have a contradiction, that is to say $u_{0}''+u_{0}$ is the some of at most two Dirac masses $I$.
\end{itemize}

To prove this estimate, we define
$v\in H_0^1(\theta_{1},\theta_{3})$ satisfying
$v''+v=\delta_{\theta_2}$ in $(\theta_{1},\theta_{3})$, $v=0$ in $\T\setminus(\theta_{1},\theta_{3})$.
In $\T$, the measure $v''+v$ is supported in $\{\theta_{1},\theta_2,\theta_{3}\}$, and since these points are in ${\Supp}(u_0''+u_{0})$, and $[\theta_{1},\theta_{3}]\subset\T_{in}$, $u_{0}+tv$ is admissible for small $|t|$.
The second order optimality condition and then the assumption (\ref{eq:coercivite}) together with Lemma \ref{lem:poincare} lead to
\begin{eqnarray*}
0\leq
j''(u_0)(v,v)&\leq&-\alpha|v|^2_{H^1(\T)} +\gamma|v|_{H^{1}(\T)}\|v\|_{H^{s}(\T)} +\beta\|v\|^2_{H^{s}(\T)}\\
&\leq&
(-\alpha+C\gamma X+C^2\beta X^2)|v|^2_{H^1(\mathbb T)},
\end{eqnarray*}
where $X=(\theta_{3}-\theta_{1})^{1-s}$, which implies \eqref{eq:ecart}
when $\beta$ is positive, and gives a contradiction if $\beta=\gamma=0$.
\end{remark}
\begin{remark}
When one uses the parametrization of convex sets by the gauge function $u$, $\Om_{u}$ is a polygon if and only if $u''+u$ is a sum of Dirac masses. With the support function (see Section \ref{sect:support}), one has the same characterization. Therefore, the conclusion is the same if we work with the optimization problem \eqref{eq:minh}. Estimate \eqref{eq:ecart} remains valid. However, $\theta_{i}$ is no longer the polar angle of a corner of the shape, 
but is the angle of the normal vectors to the successive segments of the polygonal boundary of the shape.
\hfill$\Box$
\end{remark}

As in Section \ref{sect:reg}, one can also handle problem with the equality constraint.
\subsubsection*{Proof of Theorem \ref{th:conc+constraint}.}
We now need an abstract result for second order optimality conditions. Adapting \cite[Proposition 3.3]{LN} similarly to the first order condition given in Proposition \ref{prop:order1} (this explains the assumption $j'(u_{0})\in( C^0(\T))'$), we get that
there exist $\zeta_0 \in W^{1,\infty}(\T)$ nonnegative, $\mu\in\R^d$ such that
\begin{equation}\label{eq:zeta_0-bis}
\left\{
\begin{array}{l}
\zeta_0=0\;\;\textrm{on}\;\;{\Supp}(u_0''+u_{0})\;\;\mbox{ and }\\[2mm]
\forall\; v\in W^{1,\infty}(\T_{in}),\quad j'(u_0)v +\mu\cdot m'(u_0)v=\displaystyle{
\langle\zeta_0+\zeta_0'',v\rangle_{(W^{1,\infty})'\times W^{1,\infty}}}.
\end{array}
\right.
\end{equation}
Furthermore, 
for all $v\in H^1(\T_{in})$ such that $\exists\lambda\in\R,$ with
$v''+v\geq \lambda(u_0''+u_0)$, and 
$\langle\zeta_0+\zeta_0'',v\rangle-\mu \cdot m'(u_0)(v)=0$,
\begin{equation}\label{eq:ordre2}
j''(u_0)(v,v)+\mu \cdot m''(u_0)(v,v)\geq 0.
\end{equation}
Then we proceed as in the proof of \cite[Theorem 2.1]{LN}. Compared to the first step of the proof of Theorem \ref{th:conc}, we add one degree of freedom introducing 4 functions $v_{n,i}$ on a partition of $(0,\eps_{n})$, and we look for $\lambda_{n,i}, i=1\ldots 4$ such that $v_{n}=\sum_{i=1,4}\lambda_{n,i}v_{n,i}$ satisfies
$$
v'_n(0^+)=v'_n(\eps_n^-)=\mu \cdot m'(u_{0})v_{n}=0.$$
Such a choice of $\lambda_{n,i}$ is always possible as $\lambda_{n,i}$ satisfy three linear equations. Moreover, $v_n$ is not zero and using \eqref{eq:zeta_0-bis}, we get $\int_{\mathbb T}v_n(\zeta_0+\zeta_0'')=0$, which implies
\[
0=j'(u_0)(v_n)=\int_{\mathbb T}v_n(\zeta_0+\zeta_0'')=\mu \cdot m'(u_0)(v_n).
\]
As $v_n''+v_n \geq \lambda (u_0''+u_0)$ for $\lambda\ll 0$, it follows that $v_n$ is eligible for the second order necessary condition \eqref{eq:ordre2}. Then, it follows
\begin{eqnarray*}
0&\leq&
j''(u_0)(v_n,v_n)+\mu \cdot m''(u_0)(v_n,v_n)\leq
-\alpha|v_{n}|^2_{H^1(\T)} + \gamma|v_{n}|_{H^1(\T)}\|v_{n}\|_{H^{s}(\T)}+(\beta+\|\beta'\mu\|)\|v_{n}\|^2_{H^{s}(\T)}\\
&\leq&
(-\alpha+C\gamma\eps_{n}^{1-s}+C^2(\beta+\|\beta'\mu\|)(\eps_{n})^{2(1-s)})|v_{n}|^2_{H^1(\mathbb T)}\label{eq:estimate3}
\end{eqnarray*}
As $n$ tends to $\infty$, the inequality $0\leq j''(u_0)(v_n,v_n)+\mu\cdot m''(u_0)(v_n,v_n)$ becomes impossible and this concludes the proof.\qed
\begin{remark}
An estimate similar to the one in Remark \ref{rk:corner} is not straightforward anymore, since the Lagrange multiplier $\mu$ is unknown.
\end{remark}

\subsection{Examples}\label{sect:examples-polygon}

We analyze the same examples as in Section \ref{sect:examples-smooth}, with $-P$ instead of $P$:

\begin{example}[Negative perimeter penalization]
One can study 
\begin{equation}\label{eq:ex'1}
\min \{J(\Om):=F(|\Om|,E_{f}(\Om),\la_{1}(\Om))-P(\Om)\;/\;\Om\textrm{ convex}, D_{1}\subset\Om\subset D_{2}\}
\end{equation}
where $F:(0,+\infty)\times(-\infty,0)\times(0,+\infty)\to\R$ is $ C^2$, $f\in H^2(\R^2)$, and $D_{1},D_{2}$ are bounded open sets. We can prove the following.

\begin{proposition}\label{prop:ex'1}
If $\Om_{0}$ is an optimal set for the problem \eqref{eq:ex'1}, then each connected component of the free boundary $\partial\Om_{0}\setminus( \partial D_{1}\cup\partial D_{2})$ is polygonal.
\end{proposition}
\begin{proof}
The proof is a direct consequence of Corollary \ref{cor:conc} and of the estimates given in Section \ref{ssect:dir2}. Indeed, Proposition \ref{prop:dir2} for $E_f(\Omega)$, the similar result for $\lambda_{1}$ (See Section \ref{ssect:la1''}) and Proposition \ref{p:a'',p''} for the volume, imply
$$|r''(0)|\leq C \|v\|_{H^{1/2+\eps}(\T)}^2,$$
where $r(t)=F(|\Om_{t}|,E_{f}(\Om_{t}),\la_{1}(\Om_{t}))$,  $\Om_{t}=\Om_{u_0+tv}$ and $\eps\in (0,\frac{1}{2})$. Next, the estimate for the perimeter in Proposition \ref{p:a'',p''} provides the concavity condition. 
\end{proof}

\begin{remark}
As in Remark \ref{rk:constraint}, if we consider problems of type \eqref{eq:ex'1} where the constraint $D_{1}\subset\Om\subset D_{2}$ can be dropped, then the solution is a polygon.
\end{remark}
\end{example}

\begin{example}[Volume constraint and negative perimeter penalization]
We can also consider a similar problem with a volume constraint:
\begin{equation}\label{eq:ex'2}
\min \{J(\Om):=F(E_{f}(\Om),\la_{1}(\Om))-P(\Om)\;/\;\Om\textrm{ convex}, \textrm{ and }|\Om|=V_{0}\}
\end{equation}
where $V_{0}\in(0,+\infty)$. Again, Corollary \ref{cor:conc} applies and leads to the fact that any optimal shape of \eqref{eq:ex'2} is a polygon.
\end{example}

\begin{example}\label{ex:ex'3}[Perimeter constraint]
We consider again a problem with a perimeter constraint, as in Example \ref{ex:perimeter}
\begin{equation}\label{eq:ex'3}
\min \{J(\Om):=F(|\Om|,E_{f}(\Om),\la_{1}(\Om))\;/\;\Om\textrm{ convex}, \textrm{ and }P(\Om)=P_{0}\}
\end{equation}
where $P_{0}\in(0,+\infty)$.
The optimality conditions are written for $F(|\Om|,E_{f}(\Om),\la_{1}(\Om))+\mu P(\Omega)$, where $\mu$ is a Lagrange multiplier for the constraint $P(\Omega)=P_{0}$, so if we prove that $\mu<0$, then the strategy of this section applies, and we get that any optimal shape is polygonal.

\end{example}

\subsection{Computations and estimates of second order shape derivatives}\label{sect:derivative2}

\subsubsection{Volume and perimeter}\label{sect:a'',p''}
Let $a(u)$, $p(u)$ be the area and perimeter functionals, see \eqref{e:a,p}.
\begin{proposition}\label{p:a'',p''}
Let $0<u\in W^{1,\infty}(\T)$. 
Then $a$ and $p$ are twice differentiable around $u$ in $W^{1,\infty}(\T)$ and there exists some real numbers $\beta_{1},\beta_{2},\beta_{3}, \gamma$ and $\alpha>0$ (depending on $u$)  such that, $\forall v\in W^{1,\infty}(\T)$
\begin{equation}\label{e:a'',p''}
\left\{
\begin{array}{l}
|a''(u)(v,v)|\leq \beta_{1}\|v\|^2_{L^2(\T)}\\[3mm]
\alpha|v|^2_{H^{1}(\T)}-\gamma|v|_{H^{1}(\T)}\|v\|_{L^2(\T)}-\beta_{2}\|v\|_{L^2(\T)}^2\leq p''(u)(v,v)\leq \beta_{3}\|v\|^2_{H^1(\T)}
\end{array}
\right.
\end{equation}
\end{proposition}
\begin{proof}
This is done by easy computations, using formulas of Section \ref{ssect:volper}.
\end{proof}

\subsubsection{The Dirichlet energy - Torsional rigidity}\label{ssect:dir2}

We now analyze the second order derivative of $e_f(u)=E_{f}(\Om_{u})$ introduced in Section \ref{sect:derivative1}. The main result is the following.
\begin{proposition}\label{prop:dir2}
Assume $\Om_0:=\Om_{u_0}$, $u_0>0$, $u_0''+u_0\geq0$, $f\in H^2_{loc}(\R^2)$. Then $e_{f}$ is $ C^2$ in a neighborhood of $u_0$ 
(in $W^{1,\infty}(\T)$).
Furthermore, there exist $\beta_{1},\beta_{2}$ positive such that,
for all $v\in W^{1,\infty}(\T)$, 
\begin{eqnarray}
|e_f'(u_0)v|
&\leq&
\beta_{1}\|v\|_{L^2(\T)},			\label{e:|e_f'|<}\\
|e_f''(u_0)(v,v)|
&\leq&
\beta_{2}(\|v\|_{H^{1/2}(\T)}^2+\|v\|_{L^\infty(\T)}^2).	\label{e:|e_f''|<}
\end{eqnarray}
\end{proposition}

The differentiability of $e_f$ and the estimate \eqref{e:|e_f'|<} follow easily from Proposition \ref{prop:dir}.
The estimate \eqref{e:|e_f''|<} is easy to prove when working with smooth sets and one can then even drop the $L^\infty$ term. 
However, this result is more difficult for a general convex set and the rest of this section is devoted to its proof.\\

Let $v$ be given as in Proposition \ref{prop:dir2}.
To prove the estimate \eqref{e:|e_f''|<}, it is appropriate to consider a transformation $\xi$ such that
\begin{equation}\label{e:xi(t)}
 \xi\in  C^2((-\eta,\eta),W^{1,\infty}(\R^2,\R^2)),\;\eta\in (0,1), \;\qquad
 \xi(t)=\left(\frac{1}{u_0+tv}-\frac{1}{u_0}\right)e^{i\theta}\quad \mbox{on }\ \partial\Omega_0.
\end{equation}
Then, we will differentiate twice $t\in(-\eta,\eta)\to e(t)=E(\Omega_{u_0+tv})$. We will use the following notation and identities:
\begin{equation}\label{e:*t}
\Omega_t:=\Omega_{u_0+tv},\quad
U_t:=U_{\Omega_{u_0+tv}},\quad
\hat{U}_t:=U_t\circ(I+\xi(t)),\quad
e(t):=E(\Omega_{u_0+tv}).
\end{equation}
Note that $e(t)=e_f(u_0+tv)=\mathcal E_{f,\Omega_0}(\xi(t))$ and we have
\begin{eqnarray}
e'(0)
=
e_f'(u_0)(v)
&=&
\mathcal E'_{f,\Omega_0}(0)(\xi'(0)),		\label{e:ef'=calE'}\\
e''(0)
=
e''_f(u_0)(v,v)
&=&
\mathcal E''_{f,\Omega_0}(0)(\xi'(0),\xi'(0))
+
\mathcal E'_{f,\Omega_0}(0)(\xi''(0)).		
\end{eqnarray}

In the smooth case, $e''(0)$ can be written in terms of boundary integrals, which involve in particular the boundary trace of $D^2U_0$ and $\nabla U'_0$.
These terms are not well defined in the non-smooth setting (even in the case $\Omega_0$ convex).
To overcome this difficulty, our strategy will be to write all non-smooth terms of $e''(0)$ as ``interior'' integrals in $\Om_{0}$.\\ 

\noindent
{\bf Estimate of  $e''(0)$:}
\noindent
Note that we have proven in Section \ref{sect:derivative1} that $e_f$ is $ C^2$ if $f\in H^2_{loc}(\R^2)$) (so, $e$ is $ C^2$).
We remind the following classical formulation of $e''(0)$ 
\begin{lemma}\label{l:e''(0)=}
Let $f\in H^2_{loc}(\mathbb R^2)$ and
$\xi\in  C^2(\R;W^{1,\infty}(\R^2;\R^2))$ near $0$.
Then we have
\begin{equation}\label{e:e''(0)=}
e''(0)
=
-\frac{1}{2}
\left(
\int_{\Omega_0} f U_0''  
+
\int_{\partial\Omega_0} f U_0' (\xi'(0)\cdot\nu_0)
\right),
\end{equation}
where $\hat{U}_0'':=\hat{U}_u''(u_0)(v,v)$ and $U_0''$ is defined by 
\begin{equation}\label{e:U''}
U''_0
:=
\hat{U}_0'' -
\left(
2\nabla U_0'\cdot \xi'(0) + 
\xi'(0)\cdot D^2U_0\cdot\xi'(0) + 
\nabla U_0\cdot \xi''(0)
\right)\quad\textrm{ in }\;\;\Omega_0,
\end{equation}
and satisfies
\begin{equation}
U_0''\in L^2(\Omega_0),\quad
\Delta U_0''=0\;\; \textrm{ in }\; \mathcal D'(\Omega_0).		\label{e:hU0''}
\end{equation}
\end{lemma}
{\bf Proof}.
Differentiating \eqref{e:ef'(u)(v)} at $u=u_0$ (see Corollaire 5.2.5, \cite{HP})
and then using \eqref{e:S0',S0''} gives
\begin{eqnarray}
e''(0)
&=&
-\frac{1}{2}\int_{\Om_0}
\left(
\hat{U}_0''-
2\nabla\hat{U}'_0\cdot\xi'(0)+
\xi'(0)\cdot D^2U_0\cdot\xi'(0)+
\nabla\hat{U}_0\cdot(2\nabla\xi'(0)\cdot\xi'(0)-\xi''(0))
\right)f,	
\nonumber\\
&&
-\frac{1}{2}\int_{\partial\Omega_0}U'_0f(\xi'(0)\cdot\nu_0).
\label{e:ef''-hats}
\end{eqnarray}
After replacing $\hat{U}_0'=U_0'-\nabla U_0\cdot\xi'(0)$, (\ref{e:ef''-hats}) gives
\eqref{e:e''(0)=}.

Clearly $U_0''\in L^2(\Omega_0)$. To prove that $\Delta U_0''=0$ we differentiate
\eqref{e:inthuLapphi'} at $u=u_0$ and use \eqref{e:S0',S0''}. Then we obtain
\begin{eqnarray*}
\int_{\Om_0}
\left(
\hat{U}_0''-
2\nabla\hat{U}_0'\cdot\xi'(0)+
\xi'(0)\cdot D^2U_0\cdot\xi'(0)+
\nabla U_0\cdot(2 \nabla\xi'(0)\cdot\xi'(0)-\xi''(0))
\right)
\Delta\varphi
=0.
\end{eqnarray*}
Replacing $\hat{U}_0'$ as given by \eqref{e:Ut'} gives $\int_{\Omega_0}U_0''\Delta\varphi=0$, which proves \eqref{e:hU0''}.
\hfill$\Box$\\

\subsubsection*{Proof of Proposition \ref{prop:dir2}.}
We will often write $\xi,\xi',\xi''$ for $\xi(0),\xi'(0), \xi''(0)$. Let us rewrite \eqref{e:e''(0)=} in the form  $e''(0)=\frac{1}{2}(I_1+I_2)$.
The second term $I_2$ is easy to estimate: from (\ref{e:Ut'}) we have
\begin{eqnarray}
I_2
&:=&
-
\int_{\partial\Omega_0}f U_0'(\xi'\cdot\nu_0)
=
\int_{\partial\Omega_0}f\partial_{\nu_0}U_0 (\xi'\cdot\nu_0)^2\leq C\|\xi'\|^2_{L^2(\partial\Omega)},
\quad
C=C(\|f\|_{L^\infty(\Omega_0)},\|U_0\|_{W^{1,\infty}(\Omega_0)}).	\label{e:e''-I2}
\end{eqnarray}
The first term $I_{1}=\int_{\Omega_0}U_0''\Delta U_0$ requires more investigation. To go around the non regularity of $\Omega_0$, we introduce 
$$U_0=U_1-U_2, U_i\in H^1_0(\Omega_0), -\Delta U_1=f^+,\;\;-\Delta U_2=f^-,\;\;U_i> 0\;{\rm on\;} \Omega_0.$$
Recall that $U_i\in W^{1,\infty}(\Omega_0)\cap H^2(\Omega_0)$. We will compute on the level sets $\Omega^i_\eps:=\{x\in\Omega_0, U_i(x)>\epsilon\}$ (only on one of them if $f^+\equiv 0$ or $f^-\equiv 0$). Indeed, by Sard's theorem, the $\Omega_\eps^i$  are at least $C^1$ for a.e. $\eps$. By strict positivity of $U_i$, $\lim_{\eps\to 0} \mathbbmss{1}_{\Omega_\eps^i}=\mathbbmss{1}_{\Omega_0}$, so that
$$I_1=\lim_{\eps\to 0} \int_{\Omega^1_\eps} U_0''\Delta U_1-\int_{\Omega^2_\eps}U_0''\Delta U_2.$$
Note that $U'_0, U''_0\in  C^\infty_{loc}(\Omega_0)$  and as $f\in H^2_{loc}(\mathbb R^2)$ we have  
$U_0\in H^4_{loc}(\Omega_0)$.
 We obtain
\begin{eqnarray}
 \int_{\Omega^i_\eps}U_0''\Delta U_i
 =
 \int_{\partial\Omega_\eps^i} U''_0\partial_{\nu_\eps} U_i 
 &=&
 \int_{\partial\Omega_\eps} 
 \hat{U}''_0\partial _{\nu_\eps} U_i
 -
 2(\nabla U'_0\cdot \xi')\partial_{\nu_\eps} U_i
 -
 (\xi'\cdot D^2U_0\cdot\xi')\partial_{\nu_\eps} U_i
 -
 (\nabla U_0\cdot\xi'')\partial_{\nu_\eps} U_i			\nonumber\\
&=:&
I_{1}^\eps+I_{2}^\eps + I_{3}^\eps + I_{4}^\eps.			\label{e:e''-I11234}
\end{eqnarray}
For the term $I_{1}^\eps$, we have
\begin{eqnarray}
 I_{1}^\eps
 &=&
 \int_{\Omega_\eps^i}
 \hat{U}''_0\Delta U_i + \nabla \hat{U}''_0\cdot\nabla U_i 	
 \xrightarrow{\eps\to0}
 \int_{\Omega_0}
 \hat{U}''_0\Delta U_i + \nabla \hat{U}''_0\cdot\nabla U_i 
 =
\int_{\partial\Omega_0}\hat{U}''_0\partial_{\nu_0}U_i
 = 0.							\label{e:e''-I11}
\end{eqnarray}

To deal with $I_{2}^\eps$ and $I_{3}^\eps$, we 
will need the following generalized formula of integration by parts. 
\begin{lemma}\label{l:div-th}
Let $\Omega$ be a $ C^1$ open set, $U\in W^{1,\infty}_{0}(\Omega)\cap H^2(\Omega)$,  
$V\in H^1(\Omega)\cap\{\Delta V\in L^2(\Omega_0\}$, $g\in W^{1,\infty}(\Omega;\R^2)$. 
Then
\begin{equation}\label{e:div-th)}
J:=\int_{\partial\Omega}\partial_\nu U(g\cdot\nabla V)
=
\int_{\Omega}
\nabla(\nabla U\cdot g)\cdot \nabla V
+
(\nabla U\cdot g)\Delta V
-
\nabla (\nabla^\perp U\cdot g)\cdot \nabla^\perp V,
\end{equation}
where the operator $\perp$ acts on a vector and is defined by $^\perp(a_1,a_2)=(-a_2,a_1)$. As a consequence
\begin{equation}\label{J}
|J|\leq \|\nabla U\|_{L^\infty(\Omega)}\|g\|_{L^2(\Omega)}\|\Delta V\|_{L^2(\Omega)}+2\left\{\|V\|_{H^1(\Omega)}\left[\|U\|_{H^2(\Omega)}\|g\|_{L^\infty(\Omega)}+\|\nabla U\|_{L^\infty(\Omega)}\|\nabla g\|_{L^2(\Omega)}\right].\right\}
\end{equation}
\end{lemma}
\begin{proof}
If $\nu$ is the exterior normal unit vector to $\partial\Omega$  and 
$\tau=^\perp\!\!\nu$ the unit tangent vector, then, for  $\varphi\in H^1(\Omega)$ and $a=(a_1,a_2)$, using that $\nabla U\cdot\tau=0$ on $\partial\Omega_0$,  we have
\[
  ^\perp\tau = -\nu,\quad
 (a\cdot\nu)\partial_\nu U=a\cdot\nabla U,\quad
(a\cdot\tau)\partial_\nu U=a\cdot\nabla^\perp U,\quad
 \tau\cdot\nabla\varphi=-\nu\cdot\nabla^\perp\varphi.
\]
Then we obtain
\begin{eqnarray*}
\int_{\partial\Omega}\partial_\nu U(g\cdot\nabla V)
&=& 
\int_{\partial\Omega}
(\nabla U\cdot\nu)(g\cdot\nu)(\nabla V\cdot\nu)
+
(\nabla U\cdot\nu)(g\cdot\tau)(\nabla V\cdot\tau)			\\
&=& 
\int_{\partial\Omega}
\left( 
(\nabla U\cdot g)\nabla V
-
(\nabla^\perp U\cdot g)\nabla^\perp V
\right)\cdot\nu\qquad \textrm{(apply divergence theorem to both terms)}			\\
&=& 
\int_\Omega
\nabla\cdot((\nabla U\cdot g)\nabla V)
-
\nabla\cdot((\nabla^\perp U\cdot g)\nabla^\perp V),
\end{eqnarray*}
which proves (\ref{e:div-th)}) because $\nabla\cdot\nabla^\perp=0$. The estimate (\ref{J}) follows.
\end{proof}
{\bf End of the proof of Lemma \ref{l:div-th}.}\\

\noindent
We apply Lemma \ref{l:div-th} on $\Omega=\Omega_\eps^i$ to estimate $I^\eps_2,I^\eps_3$ in (\ref{e:e''-I11234}). For $I^\epsilon_2$, we choose $U=U_i-\epsilon, V=U'_0, g=\xi'(0)$ (recall that $\Delta U'_0=0$) and for $I^\eps_3$, we choose $U=U_i-\epsilon, V=V_j=\partial_jU_i, g=g_j=\xi'_j(0)\xi'(0), j=1,2$: here $-\Delta V_j=\partial_jf^+$ or $\partial_jf^-$. Next, we apply the estimate (\ref{J}) to each of these choices and we are obviously led to estimates independent of $\eps$. For $I^\eps_4$, we make a direct easy estimate. Together also with (\ref{e:e''-I2}) and using Young inequality we obtain:
\begin{eqnarray}
\hspace*{-8mm}
 |e''(0)|
 &\leq&
 C(\|\nabla U_0'\|_{L^2(\Omega_0)}^2
 +
 \|\nabla\xi'\|_{L^2(\Omega_0)}^2+\|\xi'\|_{L^\infty(\Omega_0)}^2 + 
  \|\xi'\|_{L^2(\partial\Omega_0)}^2 +
  \|\xi''\|_{L^1(\partial\Omega_0)}),		\label{e:|e''(0)|<1}
\end{eqnarray}
\begin{equation}\label{C}
{\rm where\;}\; C=C(\|f\|_{L^\infty(\Omega_0)\cap H^1(\Omega_0)},\|U_i\|_{W^{1,\infty}(\Omega_0)},\|U_i\|_{H^2(\Omega_0)},
\, i=0,1,2).
\end{equation}

Now, let us write the estimate (\ref{e:|e''(0)|<1}) in terms of $v$.
First note that if $\alpha,\beta \in H^{1/2}(\partial\Omega_0)\cap L^\infty(\partial\Omega)$ 
then $\alpha\beta\in H^{1/2}(\partial\Omega_0)\cap L^\infty(\partial\Omega_0)$  and 
\[
\|\alpha\beta\|_{H^{1/2}(\partial\Omega_0)\cap L^\infty(\partial\Omega_0)}
\leq 
C
\|\alpha\|_{H^{1/2}(\partial\Omega_0)\cap L^\infty(\partial\Omega_0)}
\|\beta\|_{H^{1/2}(\partial\Omega_0)\cap L^\infty(\partial\Omega_0)}.
\]
(using the easy fact that $H^1(\Om_{0})\cap L^\infty(\Om_{0})$ is an algebra, and that the $H^{1/2}(\partial\Om_{0})$-norm is equivalent to the $H^1(\Om_{0})$-norm of the harmonic extension in $\Om_{0}$).
Also, we point out that the transformation 
$\psi=\psi(r,\theta):=\frac{r}{u_0(\theta)}e^{i\theta}$ is bi-Lipschitz near $\T$ and $\psi(\T)=\partial\Omega_0$. 
Then $\gamma\in H^{1/2}(\partial\Omega)$ if and only if $\gamma\circ\psi\in H^{1/2}(\T)$, and their $H^{1/2}$-norms are equivalent.

Let us remind that, according to the choice of $\xi$ in (\ref{e:xi(t)}), we have $\xi'(0)=-\frac{v}{u_0^2}e^{i\theta}$, 
$\xi''(0)=2\frac{v^2}{u_0^3}$ on $\partial\Omega_0$. 
Then we obtain, with the same dependence of the various constants $C$ as in (\ref{C})
\begin{eqnarray}\label{es1}
\|\nabla U_0'\|_{L^2(\Omega_0)}^2
\leq
C|\xi'\cdot\nabla U_0|^2_{H^{1/2}(\partial\Omega)}	
\leq
C\|\xi'\|^2_{H^{1/2}(\partial\Omega)\cap L^\infty(\Omega_0)}	
&\leq&
C\|v\|^2_{H^{1/2}(\T)\cap L^\infty(\T)},\\[2mm]			
\|\xi'\|_{L^\infty(\partial\Omega_0)}
\leq
C\|v\|_{L^\infty(\T)},	\;\;	
\|\xi'\|^2_{L^2(\partial\Omega_0)} + \|\xi''\|_{L^1(\partial\Omega_0)}
&\leq&
C\|v\|^2_{L^2(\T)}.	\label{es2}
\end{eqnarray}
All these estimates are valid for all choices of $\xi$ as in (\ref{e:xi(t)}). Let 
$${\cal W}:=\{w\in W^{1,\infty}(\Om_0),\; \;w_{|\partial\Om_0}=-\frac{v}{u_0^2}e^{i\theta}\;\}. $$
Given $w\in {\cal W}$, let us choose $\xi(t):=\zeta(t)+t(w-\zeta'(0))$, where $\zeta$ is the $W^{1,\infty}$-extension as given in \eqref{eq:xi-radial}, namely
$$\zeta(t)\left(\frac{1}{u_0(\theta)},\theta\right)=\left(\frac{1}{u_0(\theta)+tv(\theta)}-\frac{1}{u_0(\theta)}\right) e^{i\theta}\eta(r,\theta),\;\eta\in C_0^\infty(\R^2),$$
with $\eta=0$ (resp. $\eta=1$) in a neighborhood of the origin (resp. of $\partial\Omega_0$). 
Then, $\xi$ is as in (\ref{e:xi(t)}) and $\xi'(0)=w$. Therefore, the estimate (\ref{e:|e''(0)|<1}) together with (\ref{es1}), (\ref{es2}) leads to
\begin{equation}\label{e:|e''(0)|<|nablaxi|}
\forall w\in {\cal W},\; |e''(0)|
\leq
C\left(\|\nabla w\|_{L^2(\Omega_0)}^2+\|v\|^2_{H^{1/2}(\T)\cap L^\infty(\T)}+\|w\|^2_{L^\infty(\Omega_0)}\right).
\end{equation}
Let us introduce
$$w_0\in H^1(\Omega_0),\;\;\Delta w_0=0\;{\rm on}\; \Omega_0,\;\;(w_0)_{|\partial\Om_0}=-\frac{v}{u_0^2}e^{i\theta} \;\;[{\rm or}\; w_0-\zeta'(0)\in H^1_0(\Omega_0)]\;. $$
Let now $\delta_n$ be a sequence of $C_0^\infty(\Omega_0)$-functions converging to $w_0-\zeta'(0)$ in $H^1_0(\Omega_0)$ and let $w_n:=\inf\{\delta_n+\zeta'(0),\|w_0\|_{L^\infty(\Omega_0)}\}$. Then, $w_n\in {\cal W}$ and converges in $H^1(\Omega_0)$ to $w_0$. Applying (\ref{e:|e''(0)|<|nablaxi|}) with $w_n$ in place of $w$ and passing to the limit yields:
\begin{equation}\label{fin}
|e''(0)|
\leq
C\left(\|\nabla w_0\|_{L^2(\Omega_0)}^2+\|v\|^2_{H^{1/2}(\T)\cap L^\infty(\T)}+\|w_0\|^2_{L^\infty(\Omega_0)}\right).
\end{equation}
But, since $w_0$ is harmonic, 
$$\|\nabla w_0\|_{L^2(\Omega_0)}\leq \|w_0\|_{H^{1/2}(\partial\Omega_0)}\leq C\|v\|_{H^{1/2}(\T)},\;\;\|w_0\|_{L^\infty(\Omega_0)}\leq\|w_0\|_{L^\infty(\partial\Omega_0)}\leq C\|v\|_{L^\infty(\T)}.$$
Finally, the estimate (\ref{fin}) leads to
 $$|e''(0)|\leq C\|v\|^2_{H^{1/2}(\T)\cap L^\infty(\T)}.$$
\hfill$\Box$

\subsubsection{First eigenvalue of the Laplace operator with Dirichlet boundary conditions}\label{ssect:la1''}

The estimate of  Proposition \ref{prop:dir2} also holds for $\lambda_{1}(\Om_{u})$, the first Laplace eigenvalue (see Section \ref{sect:derivative1}), namely
\begin{equation}\label{vp}
|l_1''(0)|\leq C\|v\|^2{H^{1/2}(\T)}\cap L^\infty(\T),
\end{equation}
where $l_1(t)=\lambda_1(\Omega_{u_0+tv})$. As the computations are very similar, we will only sketch the proof.

\subsubsection*{Proof of (\ref{vp}).}
As for $e_f$, for $v\in W^{1,\infty}(\T)$ fixed and $|t|$ small we consider $l_1(t):=\lambda_1(\Omega_t)$ and $U_t$, the first eigenvalue and  the corresponding eigenfunction of $-\Delta$ in $\Omega_t:=\Omega_{u_0+tv}$.
As in Lemma \ref{l:e''(0)=} we can show that
\begin{equation}\label{e:l1''(0)=}
l_1''(0)=
-
\int_{\Omega_0}U_0'\Delta U_0' + U_0\Delta U_0''
=:
I_1+I_2.
\end{equation}
Here  $U_0'$ and $U_0''$ satisfy
\begin{eqnarray*}
-\Delta U_0'
=
l_1 U_0'+l_1'U_0\quad \textrm{in}\;\; \Omega_0,
\qquad
U_0'
=
-\xi'(0)\cdot\nabla U_0\;\; 
&\textrm{on}&
\;\partial\Omega_0,
\qquad
\int_{\Omega_0}U_0U_0' = 0,\\ 
-\Delta U''_0 =l_1 U_0'' + 2l_1' U_0'+l_1'' U_0
&\textrm{in}&
\;\;\Omega_0,
\qquad
\int_{\Omega_0}|U_0'|^2+U_0U_0''=0,		\\
U''_0
=
\hat{U}_0'' -\left(
2\xi'(0)\cdot\nabla U_0' + 
\xi'(0)\cdot D^2U_0\cdot\xi'(0) + \xi''(0)\cdot\nabla U_0\right)\quad 
&\textrm{in}&
\;\;\Omega_0,
\end{eqnarray*}
where $l_1=l_1(0)$, $l_1'=l_1'(0)$, $l_1''=l_1''(0)$.
Then considering $\Omega_\eps=\{x\in\Omega_0, U_0>\eps\}$ as in the proof of Proposition \ref{prop:dir2} (note that $U_0>0$ on $\Omega_0$ here), we have :
\begin{eqnarray}
 I_1
 &=&
 \int_{\Om_0}U_0'(l_1 U_0'+l_1'U_0)
 =
 l_1 \int_{\Om_0}|U_0'|^2,			\label{e:l1''-I1}\\
 I_2
 &=&
 -\lim_{\eps\to0}
 \int_{\Om_\eps}U_0\Delta U_0'' 
 =
 -\lim_{\eps\to0}
 \int_{\Om_\eps} \eps \Delta U_0''  + (U_0-\eps)\Delta U_0''	
 =
 -\lim_{\eps\to0}
  \int_{\Omega_\epsilon}(U_0-\eps)\Delta U_0''					\nonumber\\
&=&
\lim_{\eps\to0}
\int_{\Om_\eps}U_0''(-\Delta U_0)	
-
\lim_{\eps\to0}
\int_{\partial\Om_\eps}(U_0-\eps)\partial_{\nu_\eps} U_0''-U_0''\partial_{\nu_\eps}U_0		\nonumber\\
&=&
-l_1\int_{\Om_0}|U_0'|^2
+
\lim_{\eps\to0}\int_{\partial\Om_\eps}U_0''\partial_{\nu_\eps}U_0.		\label{e:l1''-I2}
\end{eqnarray}
Combining \eqref{e:l1''(0)=} with the last two equalities gives
\[
 l_1''(0)
 =
 \lim_{\eps\to0}\int_{\partial\Om_\eps}U_0''\partial_{\nu_\eps}U_0.
\]
Then we proceed exactly as in Proposition \ref{prop:dir2}, and obtain for $l_1''(0)$
an estimate exactly similar to (\ref{e:|e''(0)|<1}).

Next, we prove that 
$\|\nabla U_0'\|_{L^2(\Omega_0)}\leq C\|\xi'(0)\|_{H^1(\Omega_0)}$. As $U_0'=\hat{U}_0'-\xi'(0)\cdot\nabla U_0$ it is enough to prove 
$\|\nabla\hat{U}'_0\|_{L^2(\Omega_0)}\leq C\|\xi'(0)\|_{H^1(\Omega_0)}$.
One can verify that $\hat{U}'_0$ satisfies
\[
\hat{U}_0'\in H^1_0(\Omega_0)\cap H^2(\Omega_0),\quad
\Delta \hat{U}_0' + l_1 \hat{U}_0' =2{\rm trace}([\nabla\xi']\cdot[D^2U_0]) - l_1'U_0,\quad
\int_{\Omega_0}\hat{U}_0' U_0=0.
\]
Using the convexity of $\Omega_0$ and Fredholm alternative theorem, we can prove that the operator  
\[
V\in (H^1_0(\Omega_0)\cap H^2(\Omega_0))\backslash {\rm span}\{U_0\} 
\mapsto 
\Delta V +l_1 V \in L^2(\Omega_0)\cap\{ h,\; \int_{\Om_0} hU_0=0\},
\]
defines an isomorphism (see for example \cite{GT}), which together with the formula for $l_1'(0)$ provides the required estimate for $\hat{U}_0'$. Therefore, as for $e''(0)$, for all $\xi$ as in (\ref{e:xi(t)}), we have
\[
|l_1''(0)|
\leq
C\left(\|\nabla\xi'(0)\|_{L^2(\Omega_0)}^2+
\|v\|^2_{H^{1/2}(\T)\cap L^\infty(\T)}+\|\xi'(0)\|_{L^\infty(\Omega_0)}\right).
\]
Then we complete the proof as in Proposition \ref{prop:dir2}.
\hfill$\Box$

\section{Remarks and perspectives}

\subsection{Localization of our two approaches}\label{sect:u'3}

As explained in the introduction, the approaches leading to our two families of results are very "local" with respect to the boundary of the optimal shape. Indeed, each proof uses test functions $v\in W^{1,\infty}$ whose support may be as small as we want and only covers the portion of the boundary that we want to analyze. To show how this can be exploited, we give now -without proof- , an example of a result which can be reached by the same two methods when applied locally.

Let us consider the following optimization problem where $G:(\theta,u,q)\in\T\times\R\times\R\to\R$ is assumed to be of class $C^2$ and $a,b\in(0,\infty)$:

\begin{equation}\label{eq:min3}
\left\{\begin{array}{l} 
u_0\in W^{1,\infty}(\T),\;\;j(u_0)=\min\{j(u),\; u''+u\geq 0,\;\;a\leq u\leq b\}, \\[3mm]
\textrm{where }
j(u)=\displaystyle{\int_{\mathbb T} G\left(\theta,u(\theta),u'(\theta)\right) d\theta}.
\end{array}\right.
\end{equation}
We define $\T_{in}$ as in (\ref{eq:inside}) and we introduce the partition $\T_{in}=\T_{+}\cup \T_{0}\cup \T_{-}$ where\\ 
$\T_{+}:=\{\theta\in \T_{in}\;;\;\widetilde{G_{qq}}(\theta)\in (0,\infty)\}$, (recall $\widetilde{G_{qq}(\theta)}=\int_0^1G_{qq}\left(\theta, u_0(\theta), tu_0'(\theta^+)+(1-t)u_0'(\theta^-)\right)dt$,\\
$\T_-:=\{\theta\in\T_{in},\; [G_{qq}\left(\theta,u_0(\theta), u_0'(\theta^-)\right), G_{qq}\left(\theta, u_0(\theta), u_0'(\theta^+)\right)]\subset (-\infty,0)\}$,\\
$\T_0:=\T_{in}\setminus \left(\T_+\cup\T_-\right)$.\\
Then\\
(i) $\T_+$ is open and $u_0''\in L^\infty_{loc}(\T_+)$, so that $u_0\in W_{loc}^{2,\infty}(\T^+)$,\\
(ii) There is no accumulation point of $\Supp(u_0''+u_0)$ in  the open set $\T_-$; in other words, $[\theta \in \T_-\to \partial \Omega_{u_0}(\theta)]$ is locally polygonal.\\
The situation on $\T_{0}$ requires a complementary study specific to each functional.

\subsection{Very singular optimal shapes}

In this paper, we gave some sufficient conditions on the shape functional so that an optimal shape be smooth or polygonal.
But there exist convex sets which are not of this type, and in a certain sense have ``intermediate regularity''.
Namely, there are convex sets which are singular in the sense that they do not have corners (they are $ C^1$), but their curvature is zero almost everywhere. As an example, one may consider any convex set such that $u''+u$ is a Radon measure, without mass, but singular with respect the Lebesgue measure.

Let us mention a shape optimization problem whose solution is neither regular nor polygonal (see \cite{LM} for an analysis of this problem).
Let $\Om_{0}$ be a convex set, $V_{0}=|\Om_{0}|$, $P_{0}=P(\Om_{0})$ and $D=(\Om_{0})_{T}=\{x\in\R^2,\; d(x,\Om_{0})<T\}$. Then Theorem 8 in \cite{LM} states that:
\begin{equation}\label{eq:distance}
J(\Om_{0})=\min \{J(\Om)\;/\;\Om\subset D\textrm{ convex such that }P(\Om)=P_{0},\; |\Om|=V_{0}\},
\end{equation}
where $J$ is the distance functional:
$$J(\Om):=\int_{D}d(x,\Om)dx.$$
Since $\Omega$ is {\em any convex set}, one cannot expect any geometrical property for a minimizer of \eqref{eq:distance} without extra conditions on $D$, $V_{0}$ and $P_{0}$. Remark also that the box $D=(\Om_{0})_{T}$ is $ C^{1,1}$ here.

\subsection{Problem without perimeter}

An interesting problem, which has not been analyzed in this paper, is the following (we use the notation of Section \ref{sect:derivative2}):
\begin{equation}\label{eq:maxE}
\max \{E_{f}(\Om), |\Om|=V_{0}, \Om\textrm{ convex }\subset D\}.
\end{equation}
It is easy to prove the existence of an optimal shape $\Om_{0}$.
In this situation, we expect the term $E_{f}(\Om)$ to be leading over $|\Omega|$ (whereas the perimeter was the stronger term in the examples solved in this paper). So we are naturally led to the following question : do there exist $\alpha>0$, $\beta,\gamma\geq 0$ such that 
\begin{equation}\label{eq:e''}
\forall v\in W^{1,\infty}(\T),\quad
e''(0)\geq \alpha |v|_{H^{1/2}}^2-\gamma|v|_{H^{1/2}}\|v\|_{L^2}-\beta \|v\|_{L^2}^2\; ?
\end{equation}
A consequence of such an estimate, would be that any solution of \eqref{eq:maxE} is locally polygonal inside $D$ (the same strategy as in the proof of Theorem \ref{th:conc} would provide the result, we just need to adapt Lemma \ref{lem:poincare} to $H^{1/2}$-norms). 
It is easy to prove that \eqref{eq:e''} holds if $v$ is supported by a subset of 
$\T$ which parametrizes a $ C^2$ strictly convex part of $\partial\Om_{0}$. Therefore, with the same proof as for Theorem \ref{th:conc}, we are in position to deduce that $\partial\Om_{0}\cap D$ is nowhere $ C^2$ with a positive curvature. But it is not clear whether estimate (\ref{eq:e''}) remains valid in a more general situation and, consequently, whether $\partial\Om_{0}\cap D$ is a polygon or not.

\section*{Acknowledgments}
The work of J. Lamboley and M. Pierre is part of the project ANR-09-BLAN-0037 {\it Geometric 
analysis of optimal shapes (GAOS)} financed by the French Agence Nationale de la Recherche (ANR).
A. Novruzi thanks the agency NSERC, Canada, for the financial support of this work, and ENS Cachan, Antenne de Bretagne, for the generous hospitality during his visits there.

\bibliographystyle{plain}

\end{document}